\documentclass[a4paper,12pt]{amsart}
\usepackage{amssymb,amsmath,exscale,latexsym,amsthm,graphics,mathtools,overpic,subcaption}
\usepackage{epsfig}    
\usepackage{epic}      
\usepackage{eepic}     
\usepackage[matrix,arrow,curve,cmtip]{xy} 
\usepackage{amsfonts}
\usepackage{enumerate} 
\usepackage{paralist}
\usepackage{color}
\usepackage{MnSymbol}  
\usepackage{url}
\usepackage{enumitem}
\usepackage{hyperref}

\theoremstyle{plain}
        \newtheorem{theorem}{Theorem}[section]
        \newtheorem*{theorem*}{Theorem}
        \newtheorem*{conj*}{Conjecture}
        \newtheorem{lemma}[theorem]{Lemma}
        \newtheorem{cor}[theorem]{Corollary}
        \newtheorem{prop}[theorem]{Proposition}

\theoremstyle{definition}
        \newtheorem{definition}[theorem]{Definition}

\theoremstyle{remark}
        \newtheorem{remark}[theorem]{Remark}
        \newtheorem*{remark*}{Remark}
        \newtheorem{remarks}[theorem]{Remarks}
        \newtheorem*{remarks*}{Remarks}

        \newtheorem*{Reminder}{Reminder}

\numberwithin{equation}{section}

\def\reminder #1 {{\sf #1}}
\def\hide #1 {}
\long\def\longhide #1 {}

\newcommand{\diam}  {\operatorname{diam}}

\newcommand{\id} {\operatorname{id}}
\newcommand{\R}{\mathbb{R}}      
\newcommand{\C}{\mathbb{C}}      
\newcommand{\N}{\mathbb{N}}      
\newcommand{\Z}{\mathbb{Z}}      
\newcommand{\Halb}{\mathbb{H}}   
\newcommand{\CDach}{\widehat{\mathbb{C}}}
\newcommand{\RDach}{\widehat{\mathbb{R}}}
\newcommand{\D}{\mathbb{D}}      
\newcommand{\Aut}  {\operatorname{Aut}}
\newcommand{\Stab}  {\operatorname{Stab}}

\newcommand{\KC}{\mathcal{K}}

\newcommand{\JC}{\mathcal{J}}

\newcommand{\TC}{\mathcal{T}}

\newcommand{\Sp}{S^2}

\newcommand{\crit}{\operatorname{crit}}
\newcommand{\post}{\operatorname{post}}

\newcommand{\CC}{\mathcal{C}}

\newcommand{\SC}{\mathcal{S}}

\newcommand{\bT}{\mathbf{T}}

\newcommand{\Mdomain}{{\mathcal M}}

\newcommand{\Ker}{\operatorname{Ker}}

\newcommand{\ord}{\operatorname{ord}}

\newcommand{\FC}{\mathcal{F}}

\newcommand{\img}{\operatorname{IMG}}

\newcommand{\inter}{\operatorname{int}}

\newcommand{\XOw} {X^0_{{\tt w}}}
\newcommand{\XOb} {X^0_{{\tt b}}}

\newcommand{\wt}{{\tt w}}

\newcommand{\Rist}{\operatorname{Rist}}

\newcommand{\supp}{\operatorname{supp}}

\newcommand{\growth}{\operatorname{growth}}

\newcommand{\blue}[1]{\textcolor{blue}{#1}}

\begin{document}

\title{Exponential growth of some iterated monodromy groups}
\author{Mikhail Hlushchanka}
\address{Mikhail Hlushchanka, Department of Mathematics and Logistics, Jacobs
  University Bremen, Campus Ring 1, 28759 Bremen, Germany}
\email{m.hlushchanka@jacobs-university.de}
\author{Daniel Meyer}
\address{Daniel Meyer, Department of Mathematics and Statistics, P.O.Box 35, FI-40014 University of Jyv\"{a}skyl\"{a}, Finland}
\email{dmeyermail@gmail.com}
\date{\today}
\maketitle

\begin{abstract}
  Iterated monodromy groups of postcritically-finite rational
  maps form a rich class of self-similar groups with interesting
  properties. There are examples of such groups that
  have intermediate growth, as well as examples that have exponential growth. These
  groups arise from polynomials. We show exponential growth of
  the $\img$ of several non-polynomial maps. These include
  rational maps whose Julia set is the whole sphere, rational maps with 
  Sierpi\'{n}ski carpet Julia set, and obstructed Thurston
  maps. Furthermore, we construct the first example of a
  non-renormalizable polynomial with a dendrite Julia set whose
  $\img$ has exponential growth. 
\end{abstract}


\tableofcontents

\section{Introduction}
\label{sec:introduction}

The \emph{iterated monodromy group} ($\img$) is a group that is
defined in a natural way for certain dynamical systems, like
iteration of a rational function on the Riemann sphere $\CDach$. It was defined by
Nekrashevych, see for instance \cite{BGN}, and independently by
Kameyama in \cite{Kameyama}, see also \cite{Kameyama01}. Recently, the theory of $\img$'s, and its
applications to the study of dynamical systems, has been
developed rapidly, see in particular 
\cite{Nekra,NekraIMG11}. Many important problems in complex dynamics have
been solved with the help of $\img$'s, such as the 
\emph{Hubbard twisted rabbit problem}
\cite{BarNekr_Twist}. Iterated monodromy groups are in fact
\emph{self-similar groups}, that is, they act on a certain regular rooted tree in a ``self-similar'' fashion, and frequently have some interesting
algebraic properties. Furthermore, their connection to dynamics
provides new methods for the study of self-similar groups.

Iterated monodromy groups have been best understood for
postcriti\-cally-finite polynomials (in dynamics a map is said to
be \emph{postcritically-finite} if each of its critical points
has finite orbit). In particular, Nekrashevych gave a complete description of the IMG's of postcritically-finite polynomials in terms of  \emph{automata} generating them \cite{Nekra, Nekra_Poly}.

The study of algebraic properties of IMG's plays an important role in self-similar group theory as well as dynamics. For example, the \emph{growth properties} of iterated monodromy groups have been investigated in the last decade. Recall that a finitely generated group has either \emph{polynomial}, \emph{intermediate}, or \emph{exponential growth} depending on the volume growth of balls in the Cayley graph of the group, see Section \ref{subsec:growth-groups}. It has been
known for a while that the iterated monodromy group of the
polynomial $P_1(z)=z^2+i$ is of intermediate growth, see
\cite{BuxPerez}. Note that the \emph{Julia set} of $P_1$ is a
\emph{dendrite}, see Figure \ref{fig:Julia_bux}. On the other
hand, the iterated monodromy group of the \emph{Basilica map}
$P_2(z)=z^2-1$ is of exponential growth; this was proved (for an
isomorphic group) in \cite{GriZ_Basilica}.
Note that the closures
of the Fatou components of $P_2$ containing $0$ and $-1$
intersect, see Figure \ref{fig:Julia_basilica}. Having these two
examples in mind, various people have been trying to establish connections between dynamical properties of a map and algebraic properties of its iterated monodromy group. For instance, it can be shown that any postcritically-finite polynomial of ``Basilica type'' has iterated monodromy group of exponential growth, see Theorem~\ref{thm:poly-exp-growth}. At the same time, for the \emph{airplane polynomial}, that is, the unique polynomial of the form $P_3(z) = z^2 +c_{air}$, where $c_{air}\neq 0$ is real and satisfies $P_3^3(c_{air})=c_{air}$ ($c_{air}= -1.75488\dots$), the growth of $\img(P_3)$ is unknown. Note that two distinct bounded Fatou components of $P_3$ have disjoint closures, see Figure \ref{fig:Julia_airplane}.

\begin{figure}
  \centering
  \begin{subfigure}[b]{0.45\textwidth}
    \centering    
    \begin{overpic}
    [width=4cm,tics=20,
    ]{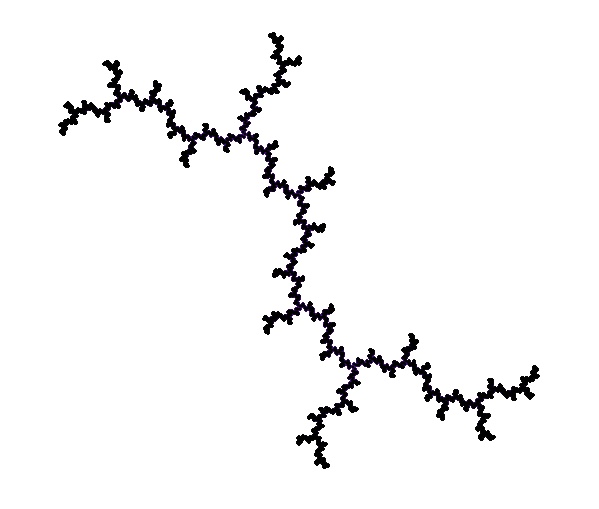}
    \put(41,43){\footnotesize $0$}
    \put(48,41){\footnotesize \blue{$\bullet$}}
    \put(52,78){\footnotesize $i$}
    \put(49,73){\footnotesize\blue{ $\bullet$}}
    \put(37,10){\footnotesize $-i$}
    \put(47,9){\footnotesize \blue{$\bullet$}}
    \put(0,78){\footnotesize $-1+i$} 
    \put(18,73){\footnotesize \blue{$\bullet$}}
  \end{overpic}
\caption{\footnotesize $P_1(z)=z^2+i$.}
 \label{fig:Julia_bux}
  \end{subfigure}
  \phantom{XXX}
  \begin{subfigure}[b]{0.45\textwidth}
    \centering
      \begin{overpic}
    [width=6cm, tics=20, 
    ]{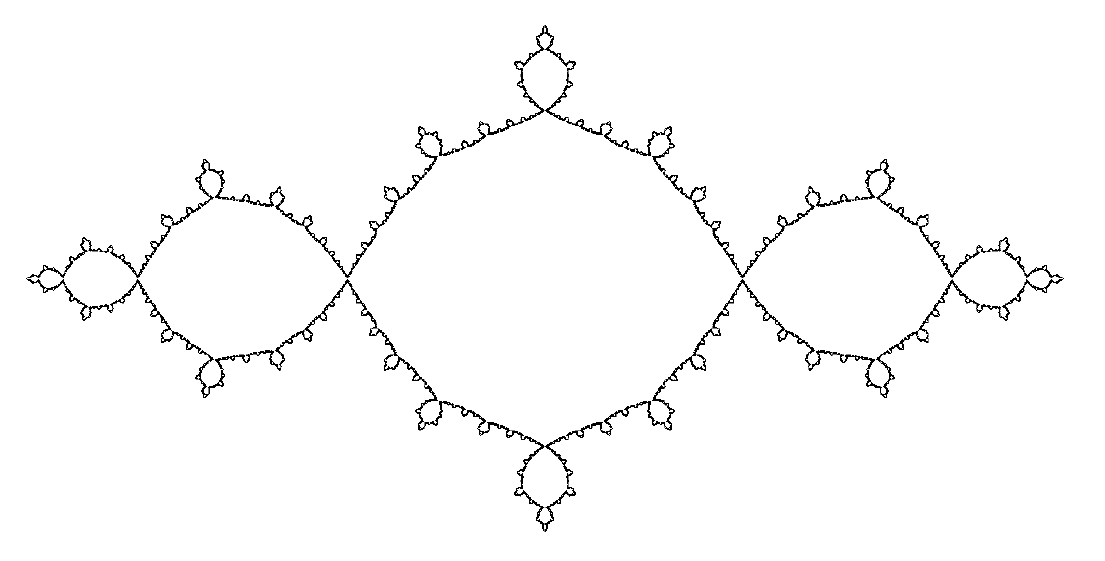} 
    \put(51.5,24){\footnotesize $0$}
    \put(48.3,24){\footnotesize \blue{$\bullet$}}
    \put(23,24){\footnotesize $-1$}
    \put(20,24){\footnotesize \blue{$\bullet$}}
  \end{overpic}
\caption{\footnotesize $P_2(z)=z^2-1$.}
\label{fig:Julia_basilica}
  \end{subfigure}
  \begin{subfigure}[b]{0.9\textwidth}
    \centering
      \begin{overpic}
    [width=10cm, tics=20, 
    ]{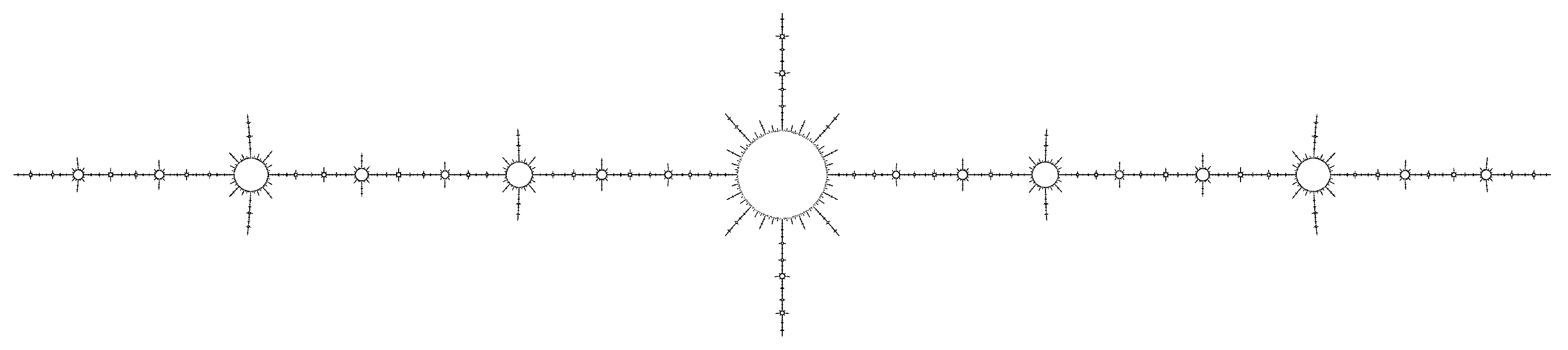} 
  \end{overpic}
\caption{\footnotesize $P_3(z)=z^2+c_{air}$.}
\label{fig:Julia_airplane}
  \end{subfigure}
  \caption{Julia sets of some quadratic polynomials.}
\end{figure}
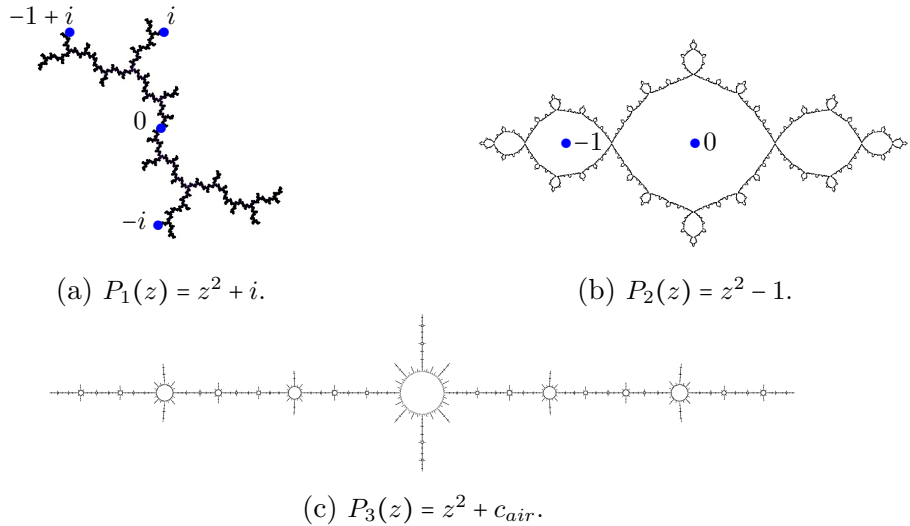

For a postcritically-finite rational map $f$ that is
not a polynomial, the iterated monodromy group is much less
understood. With the exception of \emph{Latt\`{e}s maps} (which
have iterated monodromy group that is virtually $\Z^2$) and rational maps whose Julia set ``contains a copy'' of the Julia set of a polynomial (that is, they are \emph{renormalizable} and their IMG's contain an IMG of a polynomial in a natural way), the growth of the $\img$ of a rational map $f$ has previously been unknown.

The first result we show in this paper is the following.

\begin{theorem}
  \label{thm:f1_exp_growth}
  There exists a postcritically-finite rational map $f\colon
  \CDach \to \CDach$ with Julia set equal to the whole sphere
  $\CDach$ and iterated monodromy group of exponential growth. 
\end{theorem}

We point out that this is the first non-Latt\`{e}s example of a
rational map whose Julia set is the whole Riemann sphere where the
growth is known.

The key ingredient of this paper is to use geometric
\emph{tilings} of 
the sphere $\CDach$ associated with a rational (or more
generally, branched covering) map $f$. We use them to describe
the iterated monodromy action in a simple combinatorial way (via
``rotation of tiles in flowers''). So these tilings may be viewed as
representations of the \emph{Schreier graphs} of the action of
$\img(f)$ (on the levels of the dynamical preimage tree). To
prove exponential growth of $\img(f)$, we find a free semigroup
inside it using the properties of the constructed tilings. We
also provide another proof, similar to the one in \cite{GriZ_Basilica} for the Basilica group, which uses computations with
the \emph{wreath recursions} associated with the $\img$,
see Appendix A.2.



Next, we consider some rational maps whose Julia set is a
Sierpi\'{n}ski carpet (that is, homeomorphic to the standard
Sierpi\'{n}ski carpet). In particular, this means that distinct
Fatou components have disjoint closures. \hide{In this way those rational maps are
similar to the airplane polynomial.}

\begin{theorem}
  \label{thm:fsierp_exp_growth}
  There exists a postcritically-finite rational map with Julia
  set equal to a Sierpi\'{n}ski carpet and iterated monodromy group of
  exponential growth. 
\end{theorem}

Again, this is the first example of such a rational map where the
growth of the iterated monodromy group is known.

The methods that we develop here use rather the ``combinatorial information'' about the maps than their holomorphic nature. In particular, our methods apply to some \emph{obstructed Thurston maps}, that is, postcritically-finite branched covering maps $f\colon S^2\to S^2$ that are not ``Thurston equivalent'' to a rational map, see Section \ref{sec:thurston-maps} for the definitions.

\begin{theorem}
  \label{thm:f_obtructed_exp_growth}
  There exists an obstructed Thurston map $f\colon S^2\to S^2$
  with iterated monodromy group of exponential growth. 
\end{theorem}

Since it was shown that $P_1(z)=z^2 + i$ has iterated monodromy
group of intermediate growth, it was conjectured that polynomials
that are ``similar to'' $P_1$ have $\img$ of intermediate
growth. However, the precise meaning of ``similar to'' is not
clear, and indeed has changed over time. As already noted, the Julia set
$\mathcal{J}$ of $P_1$ is a dendrite. 
Furthermore, every
finite (that is, distinct from $\infty$) point in the postcritical set $\post(P_1)=\{i, -i, -1+i,\infty\}$ is a leaf of $\mathcal{J}$, that is, does not
separate $\mathcal{J}$, see Figure \ref{fig:Julia_bux}. 
 Finally, $P_1$ is \emph{not
  renormalizable}. Roughly speaking, this means that we cannot
extract a simpler polynomial from any iterate $P_1^n$, see
Section~\ref{sec:p-not-renorm} for details. The following
conjecture is quite natural and has been studied by the community for quite some time.

\begin{conj*}\label{conj:Nekrashevych}
Let $P$ be a postcritically-finite non-renormalizable (quadratic) polynomial, such that its Julia set $\JC_P$ is a dendrite. Then $\img(P)$ is of intermediate growth. 
\end{conj*}

This conjecture is supported by quite a few examples of polynomials:  $P(z)=z^2+i$ \cite{BuxPerez}, $P(z)=z^3(-\frac{3}{2}+i\frac{\sqrt{3}}{2})+1$ \cite{FabrykowskiGuptaGrowth}, the quadratic polynomials with the kneading sequences $11(0)^\omega$ and $0(011)^\omega$ \cite{DoughertyEtAl}. The hypothesis of non-renormalizability rules out the counterexamples that arise from \emph{tuning}, a reverse operation to renormalization \cite{McM_renorm}. The tuning operation allows to construct examples of 
polynomials with dendrite Julia set whose $\img$'s are of exponential growth, see \cite[Section 5.5]{Nekra_Symb}. However, those maps will be renormalizable.

In this paper we show the following result. 

\begin{theorem}
  \label{thm:P_not_renorm_exp_growth}
  There exists a postcritically-finite polynomial $P$ with the
  following properties.
 \begin{enumerate}[label=(\arabic*),font=\normalfont]
  \item\label{cond:J_dendrite} The Julia set $\mathcal{J}$ of $P$ is a dendrite.
  \item\label{cond:post_are_leaves} Every finite postcritical point of $P$ is a leaf of
    $\JC$.
  \item\label{cond:non_renorm} $P$ is not renormalizable.
  \item\label{cond:exp_growth} The iterated monodromy group of $P$ is of exponential
    growth. 
 \end{enumerate}
\end{theorem}
 
Theorem \ref{thm:P_not_renorm_exp_growth} and its proof show that
the conjecture stated above is not valid for all polynomials,
namely for the ones of sufficiently high degree. However it may
be still true for quadratic polynomials. 
This shows that the question when the $\img$ of a polynomial is
of intermediate growth is even more subtle than previously
thought.

 
The maps in Theorems \ref{thm:f1_exp_growth}-\ref{thm:P_not_renorm_exp_growth} are explicitly constructed in a
combinatorial fashion. The proof of exponential growth relies
strongly on this combinatorial description and is essentially the
same for all considered maps. In each case it is easy to
generalize the construction to obtain infinite families of maps
with $\img$'s of exponential growth. 

Based on our examples, we provide a quite general criterion for
exponential growth of the $\img$'s of Thurston maps, see Theorem \ref{thm:exp_growth_gen}. This sufficient condition is
not the most general that we can obtain, but rather was
formulated to be easily applicable to various examples.

The paper is organized as follows. In Section \ref{sec:background}
we review some standard material about Thurston maps, complex dynamics, growth of groups, and
iterated monodromy groups. This section also contains an overview of the main properties of $\img$'s that are relevant for our paper, in particular, in the context of growth of groups. In Section~\ref{sec:constr-f} the map
$f_1$ that serves as the example in
Theorem~\ref{thm:f1_exp_growth} is constructed. 
In Section~\ref{sec:tiles-flow-iter} we review the cell
decompositions associated with $f_1$ and describe the iterated monodromy action in terms
of them. In Section~\ref{sec:exponential-growth-r} we show exponential
growth of $\img(f_1)$, that is, prove
Theorem~\ref{thm:f1_exp_growth}. In
Section~\ref{sec:crit-expon-growth} we extract from the proof of
Theorem~\ref{thm:f1_exp_growth} a general criterion for
exponential growth of $\img$'s in our setting. In Section~\ref{sec:sierp-carp-rati} we consider a
family of postcritically-finite rational maps with Sierpi\'{n}ski
carpet Julia set, in particular, we prove Theorem~\ref{thm:fsierp_exp_growth}. In
Section~\ref{sec:family-obstr-maps} we consider an obstructed
Thurston map that proves
Theorem~\ref{thm:f_obtructed_exp_growth}. In
Section~\ref{sec:non-renorm-polyn} we construct the polynomial $P$ from
Theorem~\ref{thm:P_not_renorm_exp_growth} and prove exponential
growth of its iterated monodromy group. In
Section~\ref{sec:p-not-renorm} we review renormalization theory and
prove that $P$ is not renormalizable. In Appendix A.1 we review
some standard material about actions on rooted trees and
self-similar groups. In Appendix~A.2 we give an alternative
proof of exponential growth of $\img(f_1)$, show that it is a
regular branch group, and conclude with some further properties.

\longhide{
The organization of the paper is as follows. In the next section
we review some standard material about Thurston maps, complex dynamics, growth of groups, and the
iterated monodromy group. This section also contains an overview of the main properties of $\img$'s that are relevant for our paper, in particular, in the context of growth of groups. In Section~\ref{sec:constr-f} the map
$f_1$ that serves as the example in
Theorem~\ref{thm:f1_exp_growth} is constructed. 
In Section~\ref{sec:tiles-flow-iter} we review the cell
decompositions that are defined in terms of $f_1$. This may be
viewed as iterating the combinatorial description of $f_1$ to
yield a combinatorial description of any iterate $f_1^n$ for any
$n\in \N$. We then describe the iterated monodromy action in terms
of this description. In particular, the generators of $\img(f_1)$
act by ``rotation around flowers''. So the tilings induced by
$f_1$ may be viewed as representations of the Schreier graph
acting on the $n$-th level of the dynamical preimage tree. \footnote{M: I am not sure that the paragraph about the organization of paper is the right place to say this. Maybe, it makes more sense to add a paragraph before promoting the combinatorial description of IMGs.} 
In Section~\ref{sec:exponential-growth-r} we show exponential
growth of $\img(f_1)$, that is, prove
Theorem~\ref{thm:f1_exp_growth}. In
Section~\ref{sec:crit-expon-growth} we extract from the proof of
Theorem~\ref{thm:f1_exp_growth} a general criterion for
exponential growth of $\img$'s in our setting. This condition is
not the most general, but rather was formulated to be easily
applicable. In Section~\ref{sec:sierp-carp-rati} we consider a
family of postcritically-finite rational maps with Sierpi\'{n}ski
carpet Julia, in particular, we prove Theorem~\ref{thm:fsierp_exp_growth}. In
Section~\ref{sec:family-obstr-maps} we consider obstructed
Thurston map, meaning we prove
Theorem~\ref{thm:f_obtructed_exp_growth}. In
Section~\ref{sec:p-not-renorm} we define the polynomial $P$ from
Theorem~\ref{thm:P_not_renorm_exp_growth} and prove exponential
growth of its iterated monodromy group. In
Section~\ref{sec:p-not-renorm} we review renormalization theory and
prove that $P$ is not renormalizable. 
}

\subsection{Notation}
\label{subsec:notation}
We denote by $\N$ the set of positive integers, while $\N_0$
denotes the set of non-negative integers, that is, $\N_0= \N \cup
\{0\}$. 

Let $X$ be a topological space and $U\subset X$. We denote by $\overline{U}$, $\inter(U)$, and $\partial U$ the topological closure, the interior, and the boundary of the set $U$, respectively.

The $n$-th iterate of a map $f\colon S\to S$ on a set
$S$ is denoted by $f^n$ for
$n\in \N$. We set $f^0\coloneqq \id_S$. Suppose a subset $U\subset S$ is given. We denote by
$f^{-n}(U)$ the preimage of $U$ under $f^n$, that is, $f^{-n}(U)
=\{x\in S : f^n(x)\in U\}$. For simplicity, we denote $f^{-n}(y) \coloneqq
f^{-n}(\{y\})$ for $y\in S$. Also we denote the restriction of $f$ to $U$ by
$f|U$. 

The identity element of a group is denoted by $1$. The notation ``$H<G$'' means that $H$ is a subgroup of $G$, as usual.

We consider right group actions. So if a group $G$ acts on a set $X$, then the image of $x\in X$ under the action of an element $g\in G$ is denoted by $x^g$, and in a product $g_1 g_2$ the element $g_1$ acts first, that is, $x^{g_1g_2}=(x^{g_1})^{g_2}$. We therefore write $[g_1,g_2]=g_1^{-1}g_2^{-1}g_1g_2$ and  $g_1^{g_2} = g_2^{-1}g_1g_2$.

In a group $G$, $\ord(g)$ denotes the order (or period) of a group element $g\in G$, that is, the smallest positive integer $n$ such that $g^n=1$. If no such $n$ exists, $g$ is said to have infinite order and we set $\ord(g)=\infty$.

\section{Background}
\label{sec:background}

\subsection{Thurston maps}
\label{sec:thurston-maps}

We provide a brief overview here, but refer the reader to
\cite[Chapter 2]{THEbook} for details.  Let $f\colon S^2\to S^2$ be a
\emph{branched covering map} of the topological $2$-sphere
$S^2$. That is, $f$ is a continuous and surjective map, that we
can write locally around each point $p\in\Sp$ as $z\mapsto z^d$
for some $d\in \N$ (depending on $p$) in orientation-preserving
homeomorphic coordinates in domain and target. The integer
$d \geq 1$ is uniquely determined by $f$ and $p$, and called the
\emph{local degree} of the map $f$ at $p$, denoted by
$\deg(f,p)$.

A point $c\in\Sp$ with $\deg(f,c) \geq 2$ is called a \emph{critical point} of $f$. The image of a critical point is called a \emph{critical value}. The set of all critical points of $f$ is finite and denoted by $\crit(f)$. The union $\post(f)=\bigcup_{n=1}^{\infty} {f^{n}(\crit(f))}$ of the orbits of critical points is called the \emph{postcritical set} of $f$. The map $f$ is said to be \emph{postcritically-finite} if its postcritical set $\post(f)$ is finite,  in other words,  the orbit of every critical point of $f$ is finite.

\begin{definition}
  \label{def:Thurston-map} 
  A \emph{Thurston map} is an orientation-preserving,
  post\-criti\-cally-finite, branched covering
  $f\colon S^2 \to S^2$ of topological degree $d\geq 2$.
\end{definition}

Natural examples are given by postcritically-finite
rational maps on the Riemann sphere $\CDach$ and 
postcritically-finite polynomial maps on the complex plane
$\C$. In this paper we consider Thurston maps defined by a
subdivision rule on the sphere $\Sp$.  See for example the
construction of the map $f_1$ in Section \ref{sec:constr-f}; in
greater generality this may be found in
\cite{CFKP} and \cite[Chapter~12]{THEbook}.

The \emph{ramification function}
of a Thurston map $f:\Sp\to\Sp$ is a function
$\alpha_{f} \colon \Sp \to \N\cup\{\infty\}$ such that
$\alpha_{f}(p)$ is the lowest common multiple of all local
degrees $\deg(f^n,q)$, where $q\in f^{-n}(p)$ and $n\in \N$ are
arbitrary. Thus $\alpha_{f}(p) = 1$ for all $p\in \Sp \setminus
\post(f)$. 

The \emph{orbifold} associated with $f$ is
$\mathcal{O}_f\coloneqq (S^2,\alpha_f)$. The \emph{Euler
  characteristic} of $\mathcal{O}_f$ is
\begin{equation*}
  \chi(\mathcal{O}_f) \coloneqq 2 -
  \sum_{p\in\post(f)}\left(1-\frac{1}{\alpha_f(p)}\right). 
\end{equation*}
It satisfies $\chi(\mathcal{O}_f) \leq 0$. We call $\mathcal{O}_f$
\emph{hyperbolic} if $\chi(\mathcal{O}_f)<0$, and \emph{parabolic} if $\chi(\mathcal{O}_f)=0$. A \emph{Latt\`{e}s map} is a rational Thurston
map with parabolic orbifold that does not have periodic critical
points (a point $p\in\Sp$ is called \emph{periodic} if $f^n(p)=p$ for an $n\in\N$). 

Two Thurston maps $f\colon S^2\to S^2$ and $g\colon
\widetilde{S}^2\to \widetilde{S}^2$, where $\widetilde{S}^2$ is
another topological $2$-sphere,
are called \emph{Thurston equivalent} if
there are homeomorphisms $h_0,h_1\colon S^2\to \widetilde{S}^2$
that are isotopic rel.\ $\post(f)$ such that $h_0 \circ f = g
\circ h_1$. 


\subsection{Julia and Fatou sets}
\label{subsec:Julia-Fatou-sets}

The reader is referred to \cite{Milnor_Book} for background in
complex dynamics. 

Let $f:\CDach\to\CDach$ be a rational map. Then
the \emph{Julia set} $\JC_f$ of $f$ is the closure of the set of
\emph{repelling periodic points}. If $f$ is postcritically-finite then $\JC_f$ coincides with the set of limit points of the \emph{full backwards orbit} $\bigcup_{n\geq 0} f^{-n}(t)$ of any point $t\in\CDach\setminus\post(f)$.
The \emph{Fatou set} of $f$ is the set $\FC_f=\CDach\setminus\JC_f$. A \emph{Fatou component} is a connected component of the Fatou set.


If $f$ is a polynomial, then $\JC_f$ coincides with the boundary of the set 
\[\KC_f\coloneqq\{z\in\C: \{f^n(z)\}_{n\geq 0} \text{ is bounded}\},\]
called the \emph{filled Julia set}.

\subsection{Growth of groups}
\label{subsec:growth-groups}

Given a finitely generated group $G$ with symmetric set of
generators $S$ one defines the \emph{word length} of an element $g\in G$
with respect to $S$ by
\begin{equation*}
  \ell_{G,S}(g)
  \coloneqq 
  \min\{n\in \N_0 : g=s_1\dots s_n, 
  \text{ where } 
  s_j \in S \text{ for } j=1,\dots,n\};
\end{equation*}
and the \emph{growth function} of $G$ with respect to $S$ by
\begin{equation*}
  \growth_{G,S}(n)\coloneqq \#\{g\in G : \ell_{G,S}(g) \leq n\}.
\end{equation*}

The group $G$ is said to be of
\begin{enumerate}[label=(\arabic*),font=\normalfont,leftmargin=*]
 \item\emph{polynomial growth}, if $\growth_{G,S}(n)$ is bounded above by a polynomial, that is, $\growth_{G,S}(n) \leq C n^k$ for some constants $C>0$, $k\in\N$; 
 \item\emph{exponential growth}, if $\growth_{G,S}$ is bounded below by an exponential function, that is, $\growth_{G,S}(n) \geq c \exp(\alpha n)$ for some constants $c,\alpha>0$;
 \item\emph{intermediate growth}, otherwise.
\end{enumerate}

It can be shown
that whether $G$ has polynomial, intermediate, or exponential growth does not depend on the choice of the generating set $S$.

Milnor was the first to ask whether groups
of intermediate growth exist \cite{Mil_growth}. This was answered in the positive
by Grigorchuk in \cite{Gri_G}. The example of a group of intermediate growth constructed by Grigorchuk in this paper, now called 
the \emph{Grigorchuk group}, is a \emph{self-similar group}. This
means it acts on a binary rooted tree in a certain ``self-similar'' fashion, see Definition \ref{def:self-similar_group} and \cite[Chapter~1]{Nekra}. Self-similar groups often exhibit very interesting behavior and have been
studied intensely in the last decades.  For more information on the theory of growth
of groups we refer the reader to the recent survey
\cite{Gri_growth_survey} by Grigorchuk, see also \cite{BGN}, a survey on self-similar groups and
their properties. 

\subsection{Iterated monodromy action and group}
\label{subsec:iterated-monodromy-group}

Let $f\colon S^2\to S^2$ be a Thurston map and $\post(f)$ be its postcritical set. Since $\post(f)\subset f^{-1}(\post(f))$, $f$ induces a covering 
\[
f:\Mdomain_1 = S^2\setminus  f^{-1}(\post(f)) \to \Mdomain = S^2 \setminus \post(f). 
\]

Let $d$ be the topological degree of $f$. Fix a basepoint $t\in
\Mdomain$. We consider the backward orbit of $t$, meaning the
formal disjoint union $\bT_f=\bigsqcup_{n=0}^{\infty} f^{-n}(t)$.
Then $\bT_f$ has a natural
structure of a $d$-ary rooted tree: the root, that is, the unique point on the level 0, is the basepoint
$t\in f^{-0}(t)=\{t\}$, and a vertex $p\in f^{-n}(t)$ (of the
$n$-th level) is connected by an edge to the vertex $f(p)\in
f^{-(n-1)}(t)$ (of the $(n-1)$-th level), for $n\in\N$. The set $\bT_f$ viewed as a rooted tree is called the \emph{dynamical preimage tree} of $f$ at the basepoint $t$.


The fundamental group $\pi_1(\Mdomain,t)$ acts naturally on the set of preimages $f^{-n}(t)$, $n\in\N_0$: the image $p^{[\gamma]}$ of a point $p\in f^{-n}(t)$  under the action of a loop $[\gamma]\in\pi_1(\Mdomain,t)$ is equal to the endpoint of the unique $f^n$-lift of $\gamma$ that starts at $p$. It is easy to see that the action of the fundamental group on the vertices of the dynamical preimage tree $\bT_f$ preserves the tree structure, that is, the fundamental group acts on $\bT_f$ by automorphisms of the rooted tree. Thus, we have defined a group homomorphism 
\[\phi_f:\pi_1(\Mdomain,t) \to \Aut(\bT_f)\]
from the fundamental group of $\Mdomain=\Sp\setminus\post(f)$ to the automorphism group $\Aut(\bT_f)$ of the $d$-ary rooted tree $\bT_f$.

\begin{definition}\label{def:img}
The \emph{iterated monodromy action} is the action of
$\pi_1(\Mdomain,t)$ on the dynamical preimage tree $\bT_f$. The
quotient of $\pi_1(\Mdomain,t)$ by the kernel of its action on
$\bT_f$ is called the \emph{iterated monodromy group} of $f$ and
is denoted by $\img(f)$. That is,
\[ \img(f) =  \pi_1(\Mdomain,t)/\Ker(\phi_f) \cong \phi_f(\pi_1(\Mdomain,t)) .\]
\end{definition}

\subsection{Selected properties of IMG's}\label{subsec:selected-properties-of-polynomials}
\longhide{
\begin{prop}
Let $f:\Sp\to\Sp$ be a Thurston map. Then $\img(f)$ is a self-similar group .
\end{prop}

\begin{theorem}
Let $f$ be a post-critically finite rational map. Then the sequence of the \emph{Schreier graphs} of $\img(f)$ converges to the Julia set of $f$.
\end{theorem}
}

The iterated monodromy group of a Thurston map $f$ is self-similar \cite[Proposition 5.2.2]{Nekra}. Furthermore, if $f$ is a postcritically-finite rational map, then the iterated monodromy group (together with the associated wreath recursion, see Appendix A.1) contains all the ``essential'' information about the dynamics of the map $f$: one can reconstruct from $\img(f)$ the action of $f$ on its Julia set $\JC_f$. In this case the \emph{limit space} of the iterated monodromy group is homeomorphic to the Julia set of the map \cite[Theorem 6.4.4]{Nekra}. Moreover, one can approximate the Julia set $\JC_f$ by a certain sequence of finite graphs.

\begin{definition}\label{def:Schr_graph}
Let $G$ be a group generated by a finite set $S$ and acting on a
set $X$. The \emph{labeled Schreier graph $\Gamma(G,S,X)$} is a
labeled directed (multi)graph with the set of vertices $X$ and
the set of directed edges $X\times S$, where the edge $(x,s)$
starts at $x$, ends at $x^s$, and is labeled by $s$, for each $x\in X$ and $s\in S$.
\end{definition}

Nekrashevych showed that given a postcritically-finite rational map $f$, the sequence of the Schreier graphs of
the action of $\img(f)$ on the $n$-th level of the dynamical
preimage tree $\bT_f$ converges to the Julia set of $f$,
see \cite[Chapters 3 and 6]{Nekra}.

It was observed that even very simple maps generate iterated monodromy groups with complicated structure and exotic properties which are hard to find among groups defined by more ``classical'' methods, see \cite{BGN}. For instance, $\img(z^2+i)$ is a group of intermediate growth \cite{BuxPerez} and $\img(z^2-1)$ is an amenable group of exponential growth \cite{GriZ_Basilica,BartholdiVirag}. Below we list the most important results relevant to the growth theory of $\img$'s that are known at the moment.

\begin{theorem}[\cite{NekraIMG11}]\label{thm:no-free-groups}
Let $f$ be a postcritically-finite rational map. Then $\img(f)$ does not contain a free group of rank $2$.
\end{theorem}

\begin{theorem}[\cite{NekraIMG11}]\label{thm:poly-exp-growth}
If a postcritically-finite polynomial $f$ has two finite Fatou components with intersecting closures, then $\img(f)$ contains a free semigroup of rank 2 and is of exponential growth.
\end{theorem}

\begin{theorem}[\cite{AmenabilityInBoundedGroups}]\label{thm:poly-amenable}
If $f$ is a postcritically-finite polynomial, then $\img(f)$ is amenable.
\end{theorem}

\hide{Since it was shown that $P_1(z)=z^2 + i$ has iterated monodromy
group of intermediate growth, it was conjectured that polynomials
that are ``similar to'' $P_1$ have $\img$ of intermediate
growth. However, the precise meaning of ``similar to'' is not}

For more information on the theory of iterated monodromy groups we refer the reader to \cite[Chapters 5--6]{Nekra} and \cite{NekraIMG11}.

\section{Construction of the map $f_1$}
\label{sec:constr-f}

Here we describe the map $f_1\colon \CDach \to \CDach$ that will
serve as our main example. It is a postcritically-finite
rational map, such that no critical point is periodic. This
means that the Julia set of $f_1$ is the whole Riemann sphere
$\CDach$ (see \cite[Corollary~16.5]{Milnor_Book}). Since we are mainly interested in the combinatorial
behavior of $f_1$, we will construct $f_1$ in a combinatorial
fashion. 

We consider two polyhedral surfaces $\Delta$ and $\Delta'$
constructed as follows. Let $T$ be a Euclidean triangle with angles $\pi/2,\pi/3,\pi/6$. The
surface $\Delta$ is obtained by gluing two identical copies of
$T$ together along their boundaries. The surface $\Delta'$
is obtained by gluing two Euclidean triangles with angles
$2\pi/3,\pi/6,\pi/6$ together along their boundaries. The two
triangles of $\Delta$, as well as $\Delta'$, are called the top and
bottom faces. They correspond to the top and bottom triangles in
Figure~\ref{fig:deff}. The vertices of each such triangle are
labeled $-1,1,\infty$; they correspond to the
vertices of $\Delta$ and $\Delta'$. We color the top face of
$\Delta$ white, 
and the bottom one black.  Each face of $\Delta'$ can be divided
into $6$ triangles $T'$ that are similar to $T$; we color them
black and white as shown in Figure~\ref{fig:deff}. 

The map $g_1\colon \Delta'\to \Delta$ is now constructed as
follows. Each of the $6$ white triangles $T'\subset \Delta'$ is
mapped by a similarity to the white face of $\Delta$, that is,
each vertex of $T'$ is mapped to the one of the same
angle. Similarly, each of the $6$ black triangles $T'\subset
\Delta'$ is mapped to the black face of $\Delta$ in the same
fashion. 
To illustrate
the mapping behavior, the vertices of $\Delta$ are colored red,
blue, and green in Figure~\ref{fig:deff}, and each vertex
$v$ of a triangle $T'\subset\Delta'$ is colored the same as $g_1(v)$.

\begin{figure}
  \centering
  \begin{overpic}
    [width=13cm, tics=10,
    ]{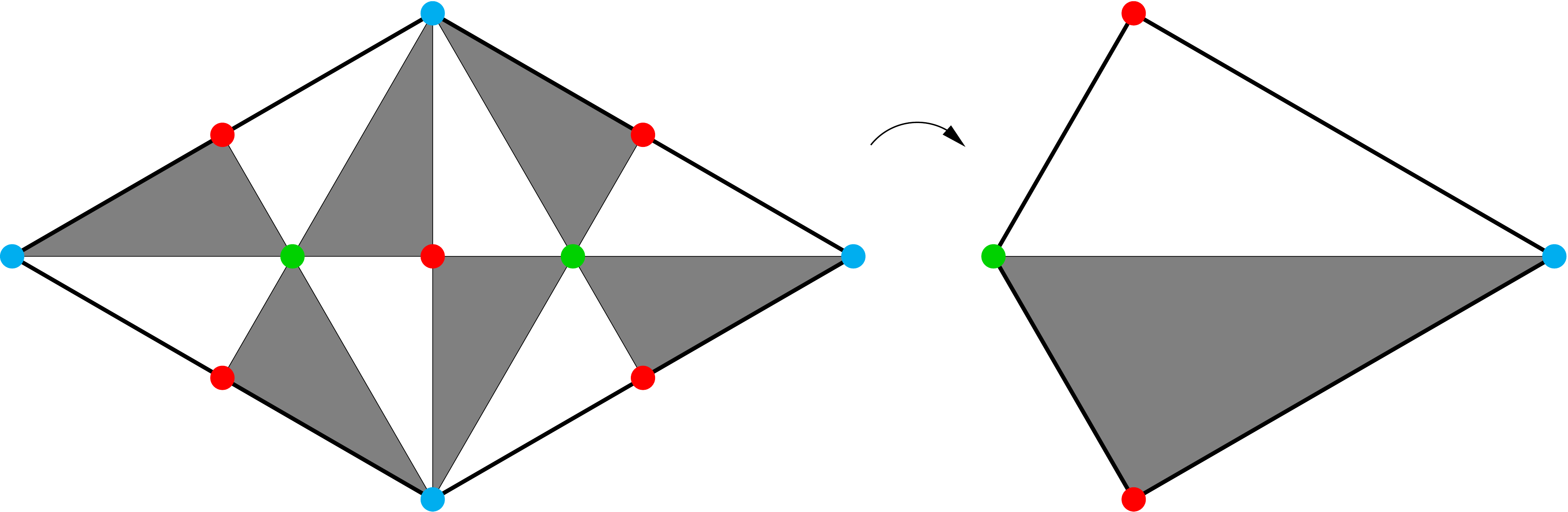}
    \put(-1.5,12.5){$-1$}
    \put(54,12.5){$1$}
    \put(29,-1){$\infty$}
    \put(29,31.5){$\infty$}
    \put(13,5.5){$a_1$}
    \put(13,25.5){$a_1$}
    \put(14.7,14.5){$c_1$}
    \put(40,5.5){$a_0$}
    \put(40,25.5){$a_0$}
    \put(25.5,13.5){$0$}
    \put(33,17){$c_0$}
    \put(8,27){$\Delta'$}
    \put(57,26.5){$g_1$}
    %
    \put(60.5,12.5){$-1$}
    \put(98.5,12.5){$1$}
    \put(73.5,-1){$\infty$}
    \put(73.5,31.5){$\infty$}
    \put(87,27){$\Delta$}
  \end{overpic}
  \caption{The polyhedral surfaces $\Delta'$ (left) and $\Delta$ (right) and the map $g_1\colon\Delta'\to\Delta$.}
  \label{fig:deff}
\end{figure}
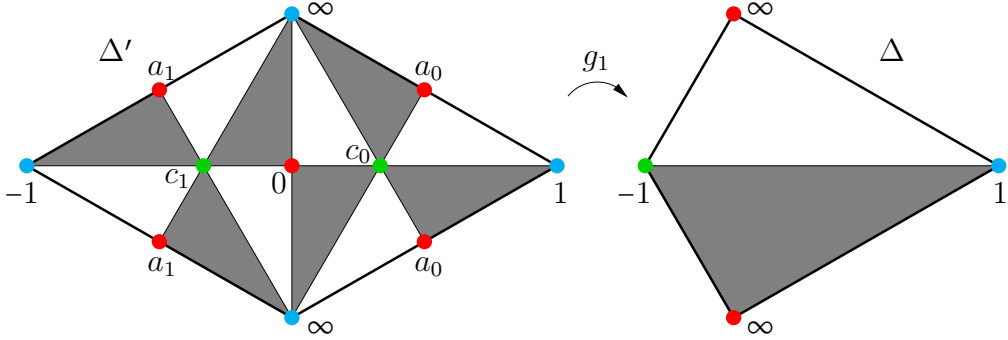

It is a standard fact that every polyhedral surface can be
equipped with a conformal structure in a natural way, see for
example \cite[Section~3.3]{BearRiem}. By the
uniformization theorem this means that there are conformal maps
$\varphi\colon \Delta\to \CDach$ and
$\varphi'\colon \Delta'\to \CDach$. To normalize these maps, we
demand that the vertices of $\Delta$ and $\Delta'$ labeled
$-1,1,\infty$ are mapped to $-1,1,\infty\in \CDach$
respectively. The symmetry then implies that the top face
of $\Delta$, as well as the top face of $\Delta'$, are mapped to
the upper half-plane $\Halb^+$; the other face of $\Delta$, and
respectively of $\Delta'$,  is mapped to the lower half-plane
$\Halb^-$ (otherwise, the map
$p\mapsto\overline{\varphi(\overline{p})}$ is distinct from
$\varphi$ violating its uniqueness\footnote{Here $\overline{p}$ denotes the point in $\Delta$ obtained from $p\in\Delta$ by reflection along the edge $[-1,1]$, see Figure~\ref{fig:deff}.}).

In the case at hand, we can actually construct the maps
$\varphi$ and $\varphi'$ explicitly. Indeed, $\varphi$ maps the
white triangle in $\Delta$ to the upper half-plane by the
Riemann map normalized so that the vertices labeled $-1,1,\infty$
are mapped to the points $-1,1,\infty\in\CDach$. Also $\varphi$ maps the black
triangle to the lower half-plane by a Riemann map with same
normalization. Similarly, $\varphi'$ is constructed by mapping
the top and bottom face of $\Delta'$ to the upper and lower half-plane with the same normalization. 


The map $f_1\colon \CDach\to \CDach$ given by
$f_1 \coloneqq \varphi \circ g_1 \circ (\varphi')^{-1}$ is now
our desired map\footnote{The map $f_1$ was originally constructed in \cite{snowemb} (where
is was called $R_6$). Its purpose then was to construct a
quasisymmetric map from a certain fractal surface to
$\CDach$. This, however, will not be relevant here.}. It is elementary to check that it is indeed a
rational map.

By a slight abuse of notation, let us denote the images of the vertices in $\Delta'$ labeled by $a_0,
a_1, c_0, c_1$ under $\varphi'$ again by $a_0, a_1, c_0, c_1$. Note
also that by symmetry, the vertex in $\Delta'$ labeled 0 is mapped to $0\in \CDach$ by $\varphi'$. 

Then the set of critical points of $f_1$ is $\crit(f_1)=\{a_0,a_1, c_0, c_1, 0,\infty\}\subset \CDach$ and the set of postcritical points is $\post(f_1)=\{-1,1,\infty\}\subset \CDach$.


Recall that the \emph{ramification portrait} of a Thurston map $f:\Sp\to\Sp$ is a directed graph with the vertex set $V=\bigcup_{n\geq 0} f^n(\crit(f))$ consisting of the union of all orbits of all critical points of $f$. For $v,w\in V$ there is a directed edge from $v$ to $w$ in the graph if and only if $f(v)=w$. Moreover, if $\deg(f,v)=d_v\geq 2$, that is, $v$ is a critical point of $f$ of degree $d_v$, we label the directed edge from $v$ to $w=f(v)$ by ``$d_v:1$''. Put
differently, the ramification portrait illustrates how critical points are mapped by the map $f$. For instance, the ramification portrait of $f_1$ is shown below.

\begin{equation}
  \label{eq:rami_f_triag_snow}
  \xymatrix{
    c_0 \ar[rd]^{3:1} & & & a_0 \ar[d]^{2:1}&
    \\
    & -1 \ar[r] & 1\ar@(ru,lu)[] & \infty \ar[l]_{4:1} & 0 \ar[l]_{2:1}
    \\
    c_1 \ar[ru]^{3:1} & & & a_1 \ar[u]_{2:1}&
  }
\end{equation}

We conclude that the
ramification function $\alpha_{f_1}$ of $f_1$ (see
Section~\ref{sec:thurston-maps}) is given by
\begin{equation*}
  \alpha_{f_1}(1) = 24, \quad \alpha_{f_1}(-1) = 3, \quad
  \alpha_{f_1}(\infty) = 2.
\end{equation*}
This means that the orbifold associated with $f_1$ is
hyperbolic. 

\begin{remark}
The map $f_1$ may be given explicitly in the following two forms:
\begin{align}
  \label{eq:defR}
  f_1(z) &= 2\left(\frac{3}{4}\right)^3 
  \frac{z^2-1}{z^2(z^2-\frac{9}{8})^2} +1
  = \frac{2(z^2 - \frac{3}{4})^3}{z^2(z^2 - \frac{9}{8})^2} -1.
\end{align}
It is elementary to check that the two expressions agree. Thus the critical points are in fact $c_0=
\frac{\sqrt{3}}{2}$, $c_1= -\frac{\sqrt{3}}{2}$, $a_0 = \frac{3}{2\sqrt{2}}$, $a_1 =
-\frac{3}{2\sqrt{2}}$ (as well as $0$ and $\infty$). Nevertheless, these
precise values and the explicit formula
\eqref{eq:defR} will be of no importance to us. 
\end{remark}

\section{Tiles, flowers, and the iterated monodromy group}
\label{sec:tiles-flow-iter}

In this section we describe tilings and the iterated monodromy action associated with a Thurston map $f:\Sp\to\Sp$ which has an $f$-invariant Jordan curve $\CC$, such that $\post(f)\subset\CC$. For simplicity, here we only discuss the case $f=f_1$. However, all the definitions and statements can be naturally adapted to the general case.

Note that the extended real line
$\widehat{\R}= \R\cup\{\infty\}\subset \CDach$ is
$f_1$-invariant, meaning that
$f_1(\widehat{\R}) \subset \widehat{\R}$, since $f_1$ is a real
function. Furthermore $\post(f_1)\subset \widehat{\R}$.  The
closures of the components of $\CDach \setminus \widehat{\R}$,
that is, of the upper and lower half-planes, are called
\emph{$0$-tiles}.  The one containing the upper half-plane is
colored white and denoted by $\XOw$; the one containing the lower
half-plane is colored black and denoted by $\XOb$.

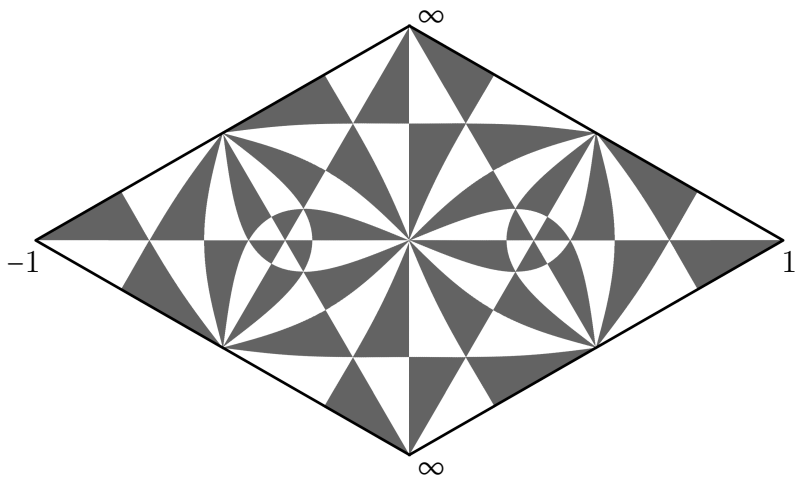
\begin{figure}
  \centering
  \begin{overpic}
    [width=10cm, tics=10, 
    ]{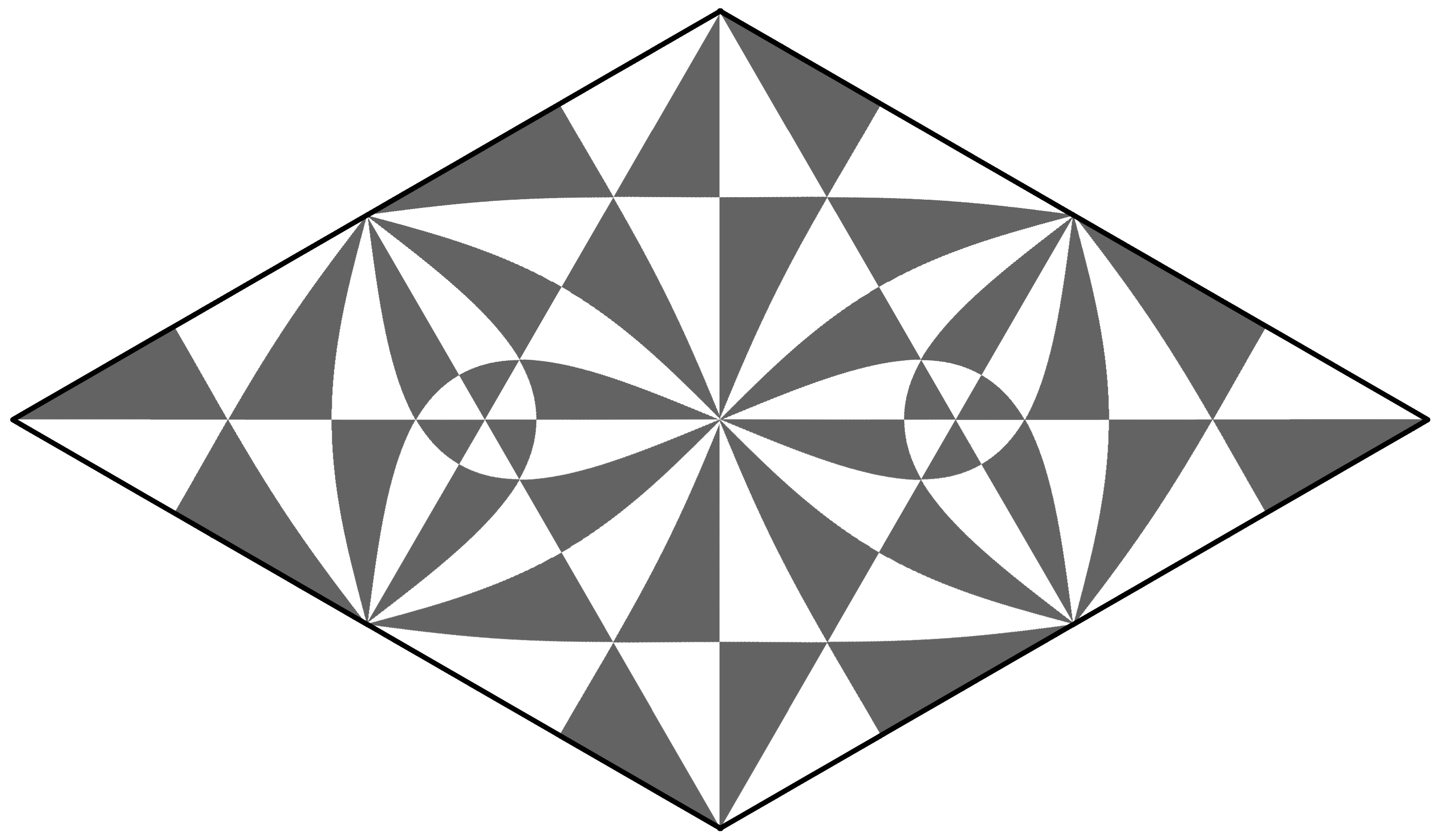}
    \put(-3,25){$-1$}
    \put(99,25){$1$}
    \put(51,-2){$\infty$}
    \put(51,58){$\infty$}
  \end{overpic}
  \caption{The $2$-tiles of $f_1$.}
  \label{fig:2tiles_f1}
\end{figure}

Let $n$ be a non-negative integer. The closure of a component of
$\CDach \setminus f_1^{-n}(\widehat{\R})$ is called an
\emph{$n$-tile}. Note that for any such
$n$-tile $X$ the set $f_1^n(X)$ is one of the two $0$-tiles. We
color $X$ by the color of $f_1^n(X)$, that is, black or
white. The \emph{subdivision} of the sphere $\CDach$ in $2$-tiles
is shown in Figure~\ref{fig:2tiles_f1} (where $\CDach$ is
identified with $\Delta'$ as in the previous section). 


Note that $f_1(\RDach)\subset \RDach$ is equivalent to
$f_1^{-1}(\RDach) \supset \RDach$, which in turn implies 
$f^{-(n+1)}(\RDach) \supset f^{-n}(\RDach)$ for all $n\in \N_0$. It 
follows that each $(n+1)$-tile is contained in an $n$-tile. For
example, when comparing Figure~\ref{fig:deff} and
Figure~\ref{fig:2tiles_f1} one sees that every $1$-tile contains
(is subdivided into) six $2$-tiles. 

Each point in $f_1^{-n}(\post(f_1))$ is called an
\emph{$n$-vertex}. Note that the set of $n$-vertices contains all
critical points of $f_1^n$. Furthermore, each postcritical point
is an $n$-vertex for any $n\in \N_0$, since
$f_1(\post(f_1)) \subset \post(f_1)$. We say that the $n$-vertex
$v$ is of \emph{type} $a$, $b$, or $c$ if $f_1^n(v) = \infty$, $1$, or $-1$,
respectively. Note that each $n$-vertex $v$ is also an
$(n+1)$-vertex, but the type of $v$ as an $n$-vertex may be
different from the type of $v$ as an $(n+1)$-vertex.

The closure of any component of $f_1^{-n}(\widehat{\R}) \setminus f_1^{-n}(\post(f_1))$ is called an \emph{$n$-edge}. Thus, the postcritical points $-1,1,\infty$ of $f_1$, which we called $0$-vertices, divide $\widehat{\R}$ into the three $0$-edges $[-1,1]$, $[1,\infty]$, and $[-\infty,-1]$.


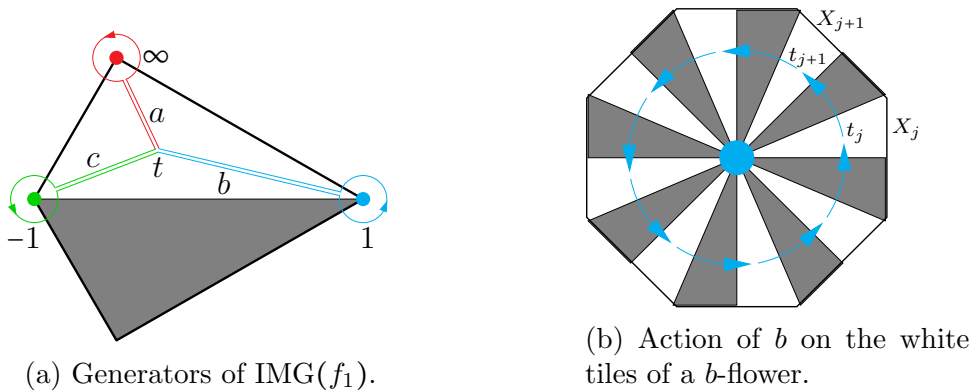
\begin{figure}
  \centering
  \begin{subfigure}[b]{0.4\textwidth}
    \begin{overpic}
      [width=5cm, tics=10,
      ]{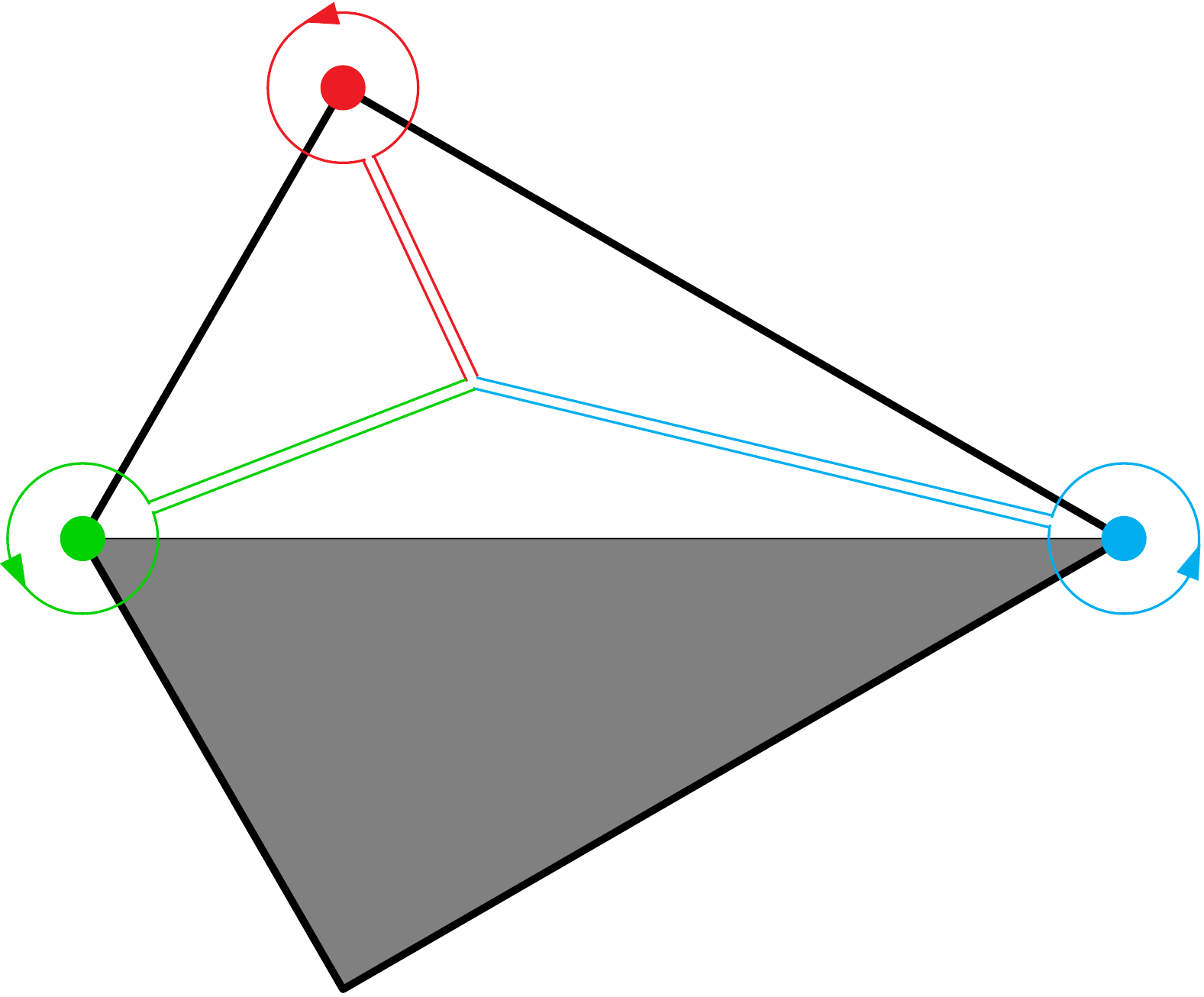}
      \put(38,44){$t$}
      \put(92.5,25){$1$}
      \put(35,74){$\infty$}
      \put(-0.5,25){$-1$}
      \put(55,39.5){$b$}
      \put(37,58.5){$a$}
      \put(20.5,46.5){$c$}
    \end{overpic}
    \caption{Generators of $\img(f_1)$.}
    \label{fig:def_gen_img}
  \end{subfigure}
  \hfill
  \begin{subfigure}[b]{0.4\textwidth}
    \begin{overpic}
      [width=4cm, tics=10,
      ]{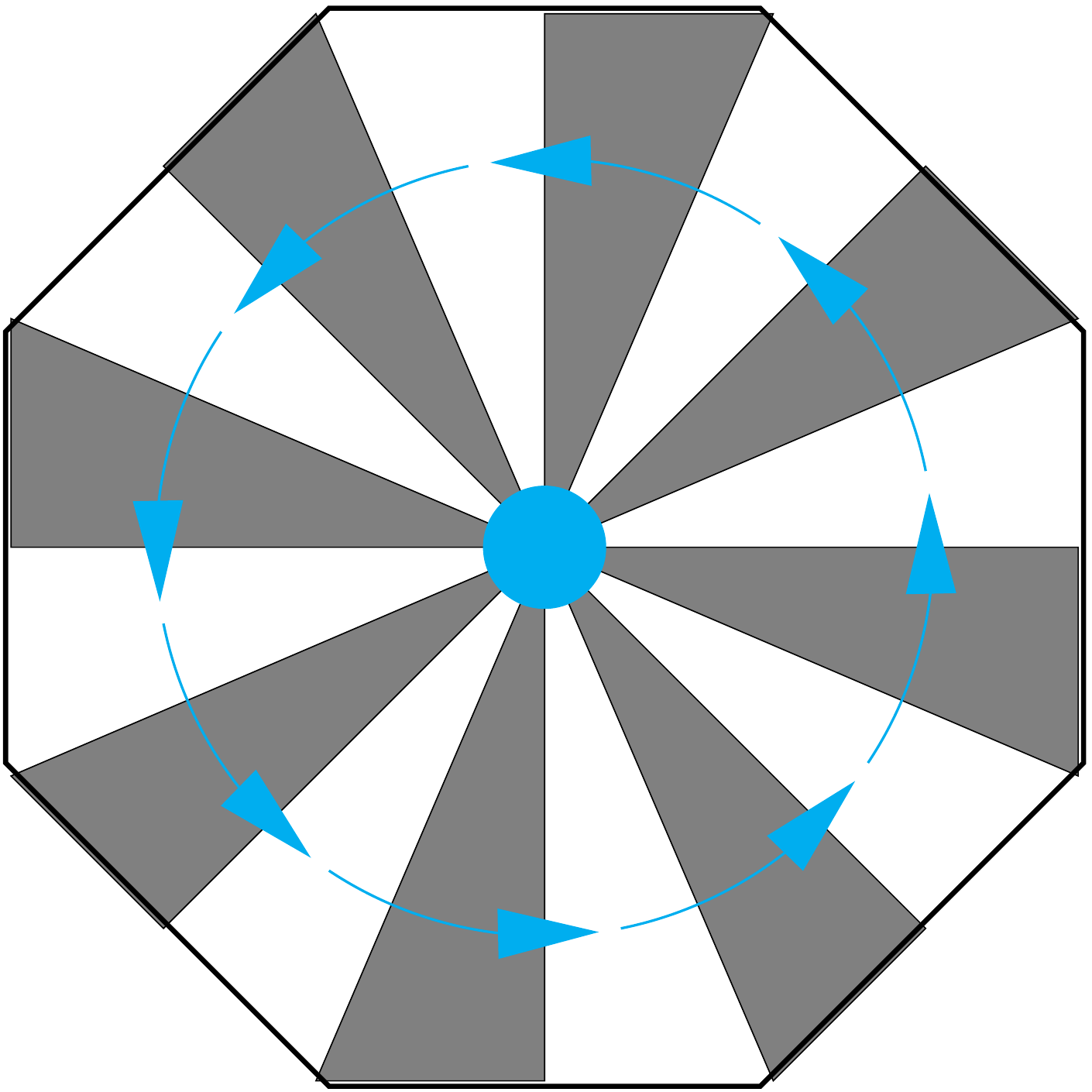}
      \put(86,57){$\scriptstyle{t_j}$}
      \put(67,82){$\scriptstyle{t_{j+1}}$}
      \put(101,58){$\scriptstyle{X_j}$}
      \put(76,94){$\scriptstyle{X_{j+1}}$}
    \end{overpic}
    \caption{Action of $b$ on the white tiles of a $b$-flower.}
    \label{fig:img_flower}
  \end{subfigure}
  \caption{The iterated monodromy action for $f_1$.}
\label{fig:img_f}
\end{figure}

Let $v\in \CDach$ be an $n$-vertex, and let $d_v:=
\deg(f^n,v)$.
Then $v$ is contained in, or incident to, $d_v$ white as well as $d_v$
black $n$-tiles. Furthermore, the colors of $n$-tiles alternate when going
around $v$. Consider the union of all such $n$-tiles,
\begin{equation*}
  \label{eq:defWnv}
  W^n(v):= \bigcup_{\substack{\text{$n$-tile $X$,}\\ \text{s.t. $v\in X$}}} X.
\end{equation*}
This set is called the \emph{flower of level $n$} centered at $v$. 
We call $d_v$ the \emph{degree} of $W^n(v)$.  We say that the flower $W^n(v)$ is an \emph{$a$-}, \emph{$b$-}, or
\emph{$c$-flower}, if the $n$-vertex $v$ is of type $a$, $b$, or
$c$, 
respectively. Note that the terminology is slightly different from the one in
\cite{THEbook}, 
where flowers are always open.
 
\smallskip Let us now choose generators of the iterated monodromy
group of $f_1$. We first choose an arbitrary basepoint $t$ in the
(interior of the) upper half-plane. In particular
$t\notin \post(f_1)$.  Let $a=\gamma_\infty$, $b=\gamma_{1}$, 
and $c=\gamma_{-1}$ be loops based at $t$ around $\infty$, $1$, and
$-1$, respectively. More precisely, we fix a small simple
positively oriented circle around $\infty$ (Here, ``small circle'' means that it is inside a neighborhood $U$ of $\infty$ such that $\post(f_1)\setminus\{\infty\} \subset \CDach \setminus U$). The loop $a$ connects
$t$ to this circle in the upper half-plane, traverses it, and
returns to $t$ in the upper half-plane. The loops $b$ and $c$ are
defined in an analogous fashion, see
Figure~\ref{fig:def_gen_img}
. By abuse of notation we identify
the equivalence classes of $a$, $b$, and $c$ in
$\img(f_1) = \pi_1(\CDach \setminus \post(f_1),t) / \Ker(\phi_{f_1})$
(see Section~\ref{subsec:iterated-monodromy-group}) with $a$, $b$, and $c$, respectively. Since $a,b,c$ generate
$\pi_1(\CDach \setminus \post(f_1),t)$, they generate
$\img(f_1)$. In fact, any two of the elements $a,b,c$ generate
$\img(f_1)$, since $acb=1$.

Note that for any white $n$-tile $X$ the map
$f_1^n\colon X \to \XOw$ is a homeomorphism, see \cite[Proposition~5.17(i)]{THEbook}. Thus each white $n$-tile contains exactly one point from
$f_1^{-n}(t)$. So we may identify the set of white $n$-tiles with the $n$-th level of the dynamical preimage tree $\bT_{f_1}$ on which $\img(f_1)$ acts. Furthermore, each white $n$-tile, and thus any $p\in f_1^{-n}(t)$, is contained in exactly one $a$-flower, exactly one $b$-flower, as well as exactly one $c$-flower of level $n$. Next we are going to describe how to read off the action of $\img(f_1)$ on tiles from the tiling picture.

Consider a $b$-flower $W^n(v)$ of level $n$. From
\eqref{eq:rami_f_triag_snow} we see that its degree is
$d_v\in\{1, 3, 4, 8\}$. Note that $\alpha_{f_1}(1)=24$ is the lowest
common multiple of these degrees by definition of the ramification function.  Let
$X_0, \dots, X_{d_v-1}$ be the white $n$-tiles contained in
$W^n(v)$ labeled mathematically positively around $v$. Fix one
such white $n$-tile $X_j \subset W^n(v)$.  Let us consider the
lift $\widetilde{b}$ of the loop $b$ starting at the point
$t_j\in f_1^{-n}(t) \cap X_j$. Then its endpoint is the unique
point $t_{j+1}\in X_{j+1}\cap f_1^{-n}(t)$ (here the index is taken
$\bmod(d_v)$). \hide{Indeed $\widetilde{b}$ crosses exactly one black
$n$-tile mathematically positively around $v$.}  We conclude that
$b$ \emph{acts on white $n$-tiles by rotations around the centers of 
  $b$-flowers}. This is illustrated in
Figure~\ref{fig:img_flower}, where blue arrows represent lifts of $b$ (up to homotopy). The analog description holds for the
generators $a$ and $c$. 

The Schreier graph for $\img(f_1)$ acting on $f^{-1}(t)$ (the
first level of the dynamical preimage tree) is shown in
Figure~\ref{fig:Schreierf1_11}. Here, we colored the edges red, blue, and greed instead of labeling them by generators $a$, $b$, and $c$, respectively.
In this way, we may think of the sphere tiling generated by the
$n$-tiles as a graphical representation of the Schreier graph of
the action of $\img(f_1)$ on the $n$-th level of the dynamical
preimage tree $\bT_{f_1}$. 

\begin{figure}
  \centering
  \begin{overpic}
    [width=7.5cm,tics=10,
    ]{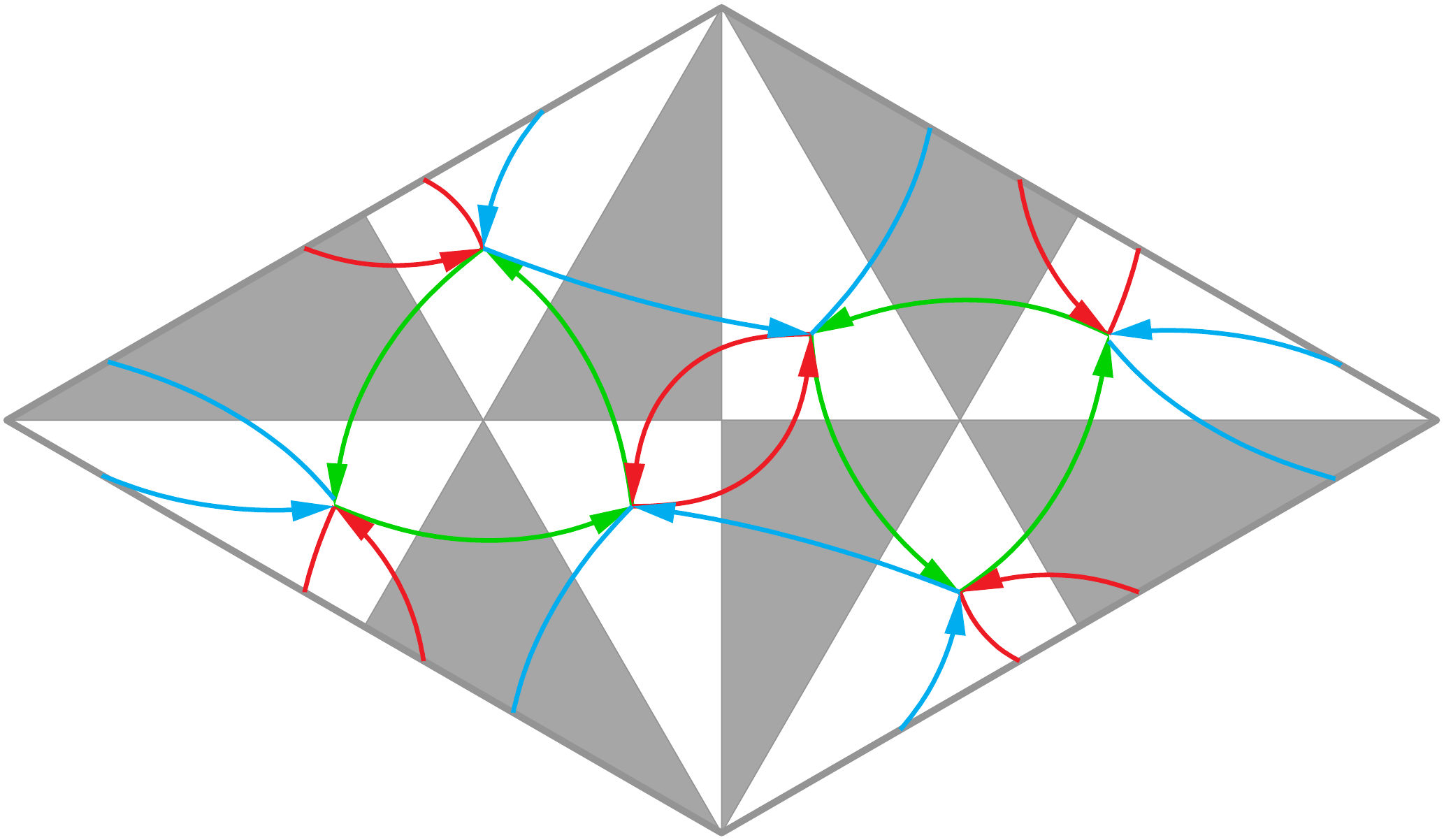}
  \end{overpic}
  \caption{The first level Schreier graph of $\img(f_1)$.}
\label{fig:Schreierf1_11}
\end{figure}

\begin{lemma}
  \label{lem:order_alpha}
  The order of each chosen generator of $\img(f_1)$ is given by the value of the
  ramification function at the corresponding postcritical point:
  \begin{equation*}
    \ord(a) = \alpha_{f_1}(\infty) = 2, 
    \quad
    \ord(b) = \alpha_{f_1}(1) = 24,
    \quad
    \ord(c) = \alpha_{f_1}(-1) = 3.
  \end{equation*}
\end{lemma}

\begin{proof}
  If $k= \alpha_{f_1}(1) =24$, then $k$ is a multiple of the
  degree of each $b$-flower. Hence $b^k$ acts trivially on each white $n$-tile, $n\in\N_0$. This means that $b^k=1$ in $\img(f_1)$. Conversely, if
  $k\geq 1$ is not a multiple of $24$, there is a $b$-flower
  $W^n(v)$ whose degree $d_v$ does not divide $k$. Then $b^k$ does not act trivially on the white $n$-tiles in $W^n(v)$, so $b^k\neq 1$. Thus $\ord(b) = \alpha_{f_1}(1) = 24$. The argument is
  completely analogous for the generators $a$ and $c$.
\end{proof}

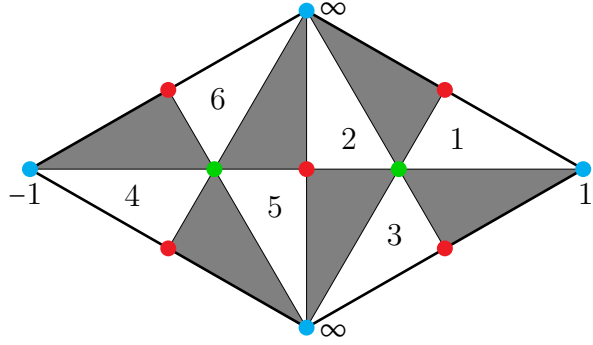
\begin{figure}
  \centering
  \begin{overpic}
    [width=7.5cm, tics=10,
    ]{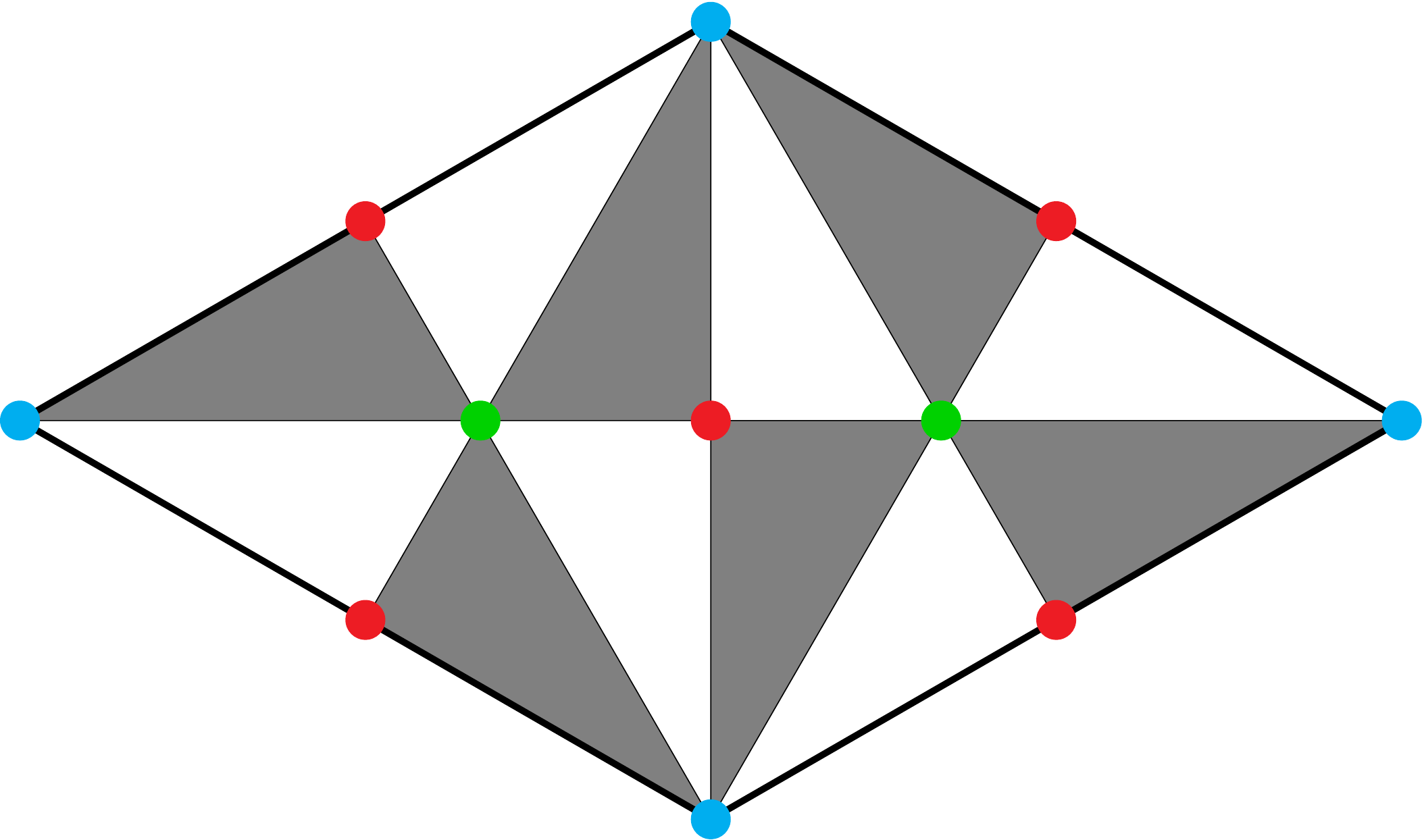} 
    \put(97.5,23.3){$1$}
    \put(52,56.5){$\infty$}
    \put(52,-.5){$\infty$}
    \put(-2.5,23.3){$-1$}
    \put(75,33){$1$}
    \put(56,33){$2$}
    \put(64,16){$3$}
    \put(18,23){$4$}
    \put(43,21){$5$}
    \put(33,40){$6$}
  \end{overpic}
  \caption{Labeling of the $1$-tiles of $f_1$.}
  \label{fig:def_label}
\end{figure}

Let us now give the \emph{wreath recursions} of the generators $a,b,c$, see Appendix A.1. To do this we
label the white $1$-tiles as
indicated in Figure~\ref{fig:def_label}. 
Let
$t_j\in f_1^{-1}(t)$ be the preimage of $t$ contained in the
$1$-tile labeled by $j\in \{1,\dots,6\}$. Note that $f_1([-\infty,-1] \cup [1,\infty])\subset [1,\infty]$, that is, $[-\infty,-1] \cup [1,\infty]$ forms a forward-invariant tree in $\CDach$ joining the postcritical points of $f_1$. This allows us to naturally choose \emph{connecting paths} from the basepoint $t$ to the points in $f_1^{-1}(t)=\{t_1,\dots,t_6\}$ and define a \emph{labeling} on the dynamical preimage tree $\bT_{f_1}$. Namely, we connect $t$ to
$t_j$, $j\in \{1,\dots,6\}$, by a path $\ell_j$ that does not intersect $[-\infty,-1] \cup
[1,\infty] \subset \widehat{\R}\subset \CDach$, meaning that $\ell_j$ stays
in the interior of the domain in
Figure~\ref{fig:def_label}. These choices define the label of every vertex of $\bT_{f_1}$ uniquely by iterative lifting of the paths $\ell_1, \dots, \ell_6$, see
\cite[Chapter 5.2]{Nekra} for more details. With respect to this labeling the action of $\img(f_1)$ on the regular rooted tree $\bT_{f_1}$ becomes self-similar, see Definition \ref{def:self-similar_group}. The wreath recursions of $a$, $b$, 
and $c$ are then given by
\begin{align}
  \label{eq:wreath_f}
  a&=\llangle b^{-1},1, b,c^{-1},1,c\rrangle \;(1 3) (2 5) (4 6)
  \\
  \notag
  b&=\llangle b, b^{-1},1, c, c^{-1}, 1 \rrangle \; (2 3 5 6)
  \\
  \notag
  c&= (1 2 3) (4 5 6).
\end{align}
We will however not need the
wreath recursions to prove exponential growth of $\img(f_1)$.


\section{Exponential growth of $\img(f_1)$}
\label{sec:exponential-growth-r}

In this section we prove that the iterated monodromy group of the
map $f_1$ has exponential growth. More precisely, we construct a
free semigroup inside $\img(f_1)$. We remind the reader that a
more classical, a la \emph{Basilica}, proof which only uses
computations with the wreath recursions \eqref{eq:wreath_f} can be
found in Appendix A.2.

Recall that the map $f_1$ has three postcritical points $-1$, $1$, and $\infty$, which we
call $0$-vertices, that divide $\widehat{\R}$ into the three $0$-edges $[-1,1]$, $[1,\infty]$, and $[-\infty,-1]$. The
reader is advised to consult Figure~\ref{fig:deff}. The $0$-edge $[1,\infty]$ will be of special importance to us.

\begin{lemma}
  \label{lem:deg_a_flower}
  Any $a$-flower of level $n\geq 1$ has degree $2$.
\end{lemma}

\begin{proof}
  From \eqref{eq:rami_f_triag_snow} we see that every preimage
  of $\infty$ is a critical point of local degree $2$, which is
  not a postcritical point. The statement follows.
\end{proof}

Note that Lemma \ref{lem:deg_a_flower} is stronger than the statement
that $\ord(a)=2$, since the latter does not rule out $a$-flowers of degree $1$. 
 
\begin{lemma}[{Flowers on $[1,\infty]$}]
  \label{lem:flowers_1_infty}
  Let $n\in\N_0$. The $0$-edge $[1,\infty]$ has the following properties.
 \begin{enumerate}[label=(\arabic*),font=\normalfont]
  \item 
    \label{item:flowers_1_inf1}
     $f_1([1,\infty]) =
    [1,\infty]$. Consequently, the $0$-edge $[1,\infty]$ is (forward)
    \emph{$f_1$-invariant}.
  \item 
    \label{item:flowers_1_inf2}
    For every $n$-vertex $v\in (1,\infty)$ of type $b$ the
    degree of $W^n(v)$ is $8$. The degree of the $b$-flower
    $W^n(1)$ is $1$.
  \item
    \label{item:flowers_1_inf3} 
   There are exactly $2^{n}+1$ $n$-vertices on $[1,\infty]$. Moreover, their type  alternates between $b$ and $a$.
\longhide{    Along $[1,\infty]$ the type of $n$-vertices alternates
    between $b$ and $a$. There are no $n$-vertices of type $c$ on
    $[1,\infty]$.} 
  \item
    \label{item:flowers_1_inf4} 
    For any $n$-vertex $v\in (1,\infty)$, the number of
    white $n$-tiles in the flower $W^n(v)$ that are
    contained in $\XOw$ equals
    the number of white $n$-tiles in $W^n(v)$ that are
    contained in $\XOb$.
    %
  \item 
    \label{item:flowers_1_inf5}
    For any $n$-vertex $v\in [1,\infty)$ there is a unique
    white $n$-tile $X(v)\in W^n(v) \cap \XOw$ that intersects
    $\widehat{\R}$ in an $n$-edge.
  \end{enumerate}
\end{lemma}

\begin{figure}
  \centering 
  \begin{overpic} 
    [width=12cm, 
    tics =10]{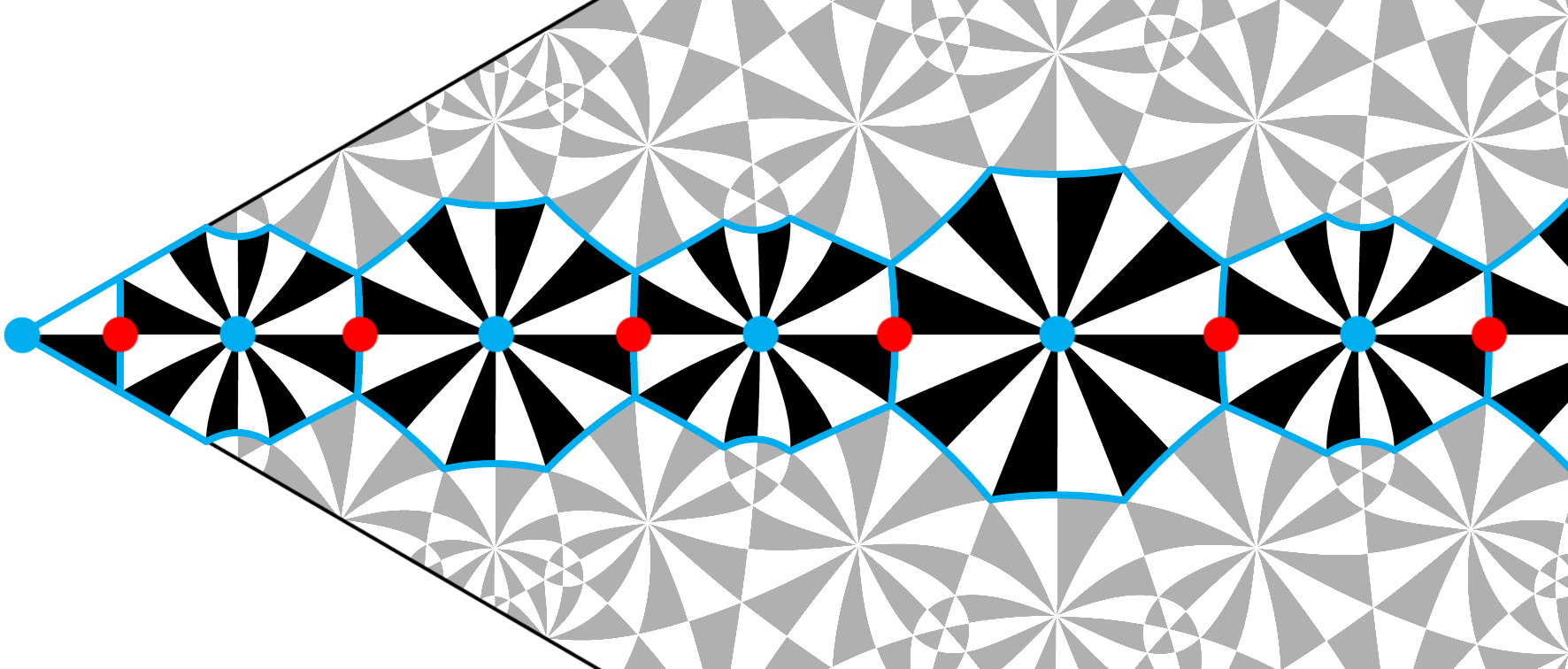}
    \put(0,17){$1$}
    \put(103,20){$\cdots \;\infty$}
  \end{overpic}
  \caption{$b$-flowers on $[1,\infty]$.}
  \label{fig:bflower_1inf}
\end{figure}

The situation is illustrated in
Figure~\ref{fig:bflower_1inf}. 
The $n$-vertices on $[1,\infty]$ of type
$b$ are marked as blue dots, the ones of type $a$ as red dots. 
Also the $b$-flowers on $[1,\infty]$ are outlined in blue.\footnote{For
purely aesthetic reasons we have cut the Riemann sphere $\CDach$
along $[-\infty, 1]$ and applied a $6$-th root, that is, the picture
shows the tiles after applying the map $z\mapsto (z-1)^{1/6}$. This
ensures that the $b$-flowers on $[1,\infty]$ have roughly the
same size.}

\longhide{
Recall that $\XOw$ is the closure of the upper half-plane,
$\XOb$ the closure of the lower half-plane. So
\eqref{item:flowers_1_inf4} says that $W^n(v)$ contains as many
white $n$-tiles above the real line as below.
It is easy to show that there are in fact $2^{n-1} +1$
$n$-vertices of type $b$ on $[1,\infty]$ and $2^{n-1}$
$n$-vertices of type $a$ on $[1,\infty]$, though this precise
number will not be relevant for us. 
}

\begin{proof}\mbox{}

  \ref{item:flowers_1_inf1} 
  Note that $[1,\infty] = [1,a_0] \cup [a_0, \infty]$, and
  $f_1$ maps $[1,a_0]$ as well as $[a_0,\infty]$ homeomorphically
  to $[1,\infty]$, see Figure~\ref{fig:deff}. 

  \smallskip \ref{item:flowers_1_inf2} Note that $(1,\infty)$
  contains no $1$-vertices of type $b$ (that is, there is no point
  $v\in (1,\infty)$ with $f(v)=1$). The only $2$-vertex of type
  $b$ on $(1,\infty)$ is $a_0$ (which satisfies $f_1(a_0)=\infty$
  and $f_1^2(a_0)=1$).  It follows from
  \ref{item:flowers_1_inf1} that for any $n$-vertex
  $v\in (1,\infty)$ of type $b$ the orbit
  $v, f_1(v), f_1^2(v),\dots, f_1^n(v)=1$ contains exactly $2$
  critical points of $f_1$, namely $a_0$ and $\infty$. Thus
  $\deg(f_1^n,v) = \deg(f_1,a_0)\deg(f_1,\infty)=8$ is the degree
  of $W^n(v)$ as desired. Clearly, $\deg(f_1^n,1) =1$ for all
  $n\in \N_0$.

  \smallskip
  \ref{item:flowers_1_inf3} The statement follows from the
  description above and an elementary induction. 

  \smallskip
  \ref{item:flowers_1_inf4}
Recall that $\XOw$ and $\XOb$ are the closures of the upper and lower half-planes, respectively. So we need to show that each flower $W^n(v)$, $v\in (1,\infty)$, contains as many
white $n$-tiles above the real line as below.
  
  First note that $f_1$ is a real
  function, meaning that $\overline{f_1(z)} =
  f_1(\overline{z})$.
  Thus the $n$-tiles are symmetric with respect to the real
  axis. \hide{More precisely, the reflection of a white $n$-tile along the real axis is a black $n$-tile and vice versa.} Since $\deg(f_1^n,v)$ is even, the number of $n$-tiles in $W^n(v)$ below and above the real line is the same even number. As colors of tiles around $v$ alternate the statement follows.


  \smallskip   
  \ref{item:flowers_1_inf5} The statement follows from the above considerations.
\end{proof}

\begin{cor}
  \label{cor:ab4_inf_order}
  $ab^4$, $ab^{12}$, and $ab^{20}$ are distinct elements of infinite order in $\img(f_1)$.
\end{cor}

\begin{proof}
Since $\ord(b)=24$ it follows that  $ab^4$, $ab^{12}$, and $ab^{20}$ are distinct in $\img(f_1)$.

Let $n\geq 2$ be an integer. Consider an $n$-vertex $v\in [1,\infty)$ of type $b$ and the white $n$-tile $X(v)$ as in
Lemma~\ref{lem:flowers_1_infty}\ref{item:flowers_1_inf5}. Let $v'$ be the $n$-vertex of type $b$ to the right
of $v$ on $[1,\infty]$ and $X(v')\subset W^n(v') \cap \XOw$ be the white
$n$-tile according to Lemma~\ref{lem:flowers_1_infty}\ref{item:flowers_1_inf5}. In the latter we assume that $v'$ is different from $\infty$. Then by the description of the iterated monodromy action on tiles from Section \ref{sec:tiles-flow-iter} as well as Lemmas~\ref{lem:deg_a_flower} and \ref{lem:flowers_1_infty} it follows that $ab^4$ maps $X(v)$
to $X(v')$. Put differently, $ab^4$ ``shifts white $n$-tiles in $\XOw$ on the $0$-edge $[1,\infty)$ to
the right''. \hide{ (and white $n$-tile in $\XOb$ to the left)} Since the degree of $W^n(v')$ is $8$, it follows that $ab^{12}$ and $ab^{20}$ act on $X(v)$ in exactly the same way as $ab^4$, that is, by shifting $X(v)$ to $X(v')$ on the right. 

Note that by Lemma~\ref{lem:flowers_1_infty}\ref{item:flowers_1_inf3} there are $2^{n-1}$ $n$-vertices of type $b$ on $[1,\infty)$. From the above considerations, the elements $ab^4$, $ab^{12}$, and $ab^{20}$ are of infinite order in $\img(f_1)$.
\end{proof}

\begin{lemma}
  \label{lem:free_semi_gp}
  The elements $ab^4$, $ab^{12}$, and $ab^{20}$ generate a free
  semigroup in $\img(f_1)$.
\end{lemma}

Put differently, we consider the words of the form
\begin{equation}
  \label{eq:words_free_semi}
  ab^{k_1} ab^{k_2} \dots ab^{k_N},
\end{equation}
where $N\in \N_0$ and $k_j\in\{4,12,20\}$ for $j=1,\dots,N$. We
will show that if two such words are distinct, then they are distinct
as elements of $\img(f_1)$.

\begin{figure}
  \centering
  \begin{subfigure}[b]{0.4\textwidth}
    \centering
    \begin{overpic}
      [width=5.5cm,
      tics=10]{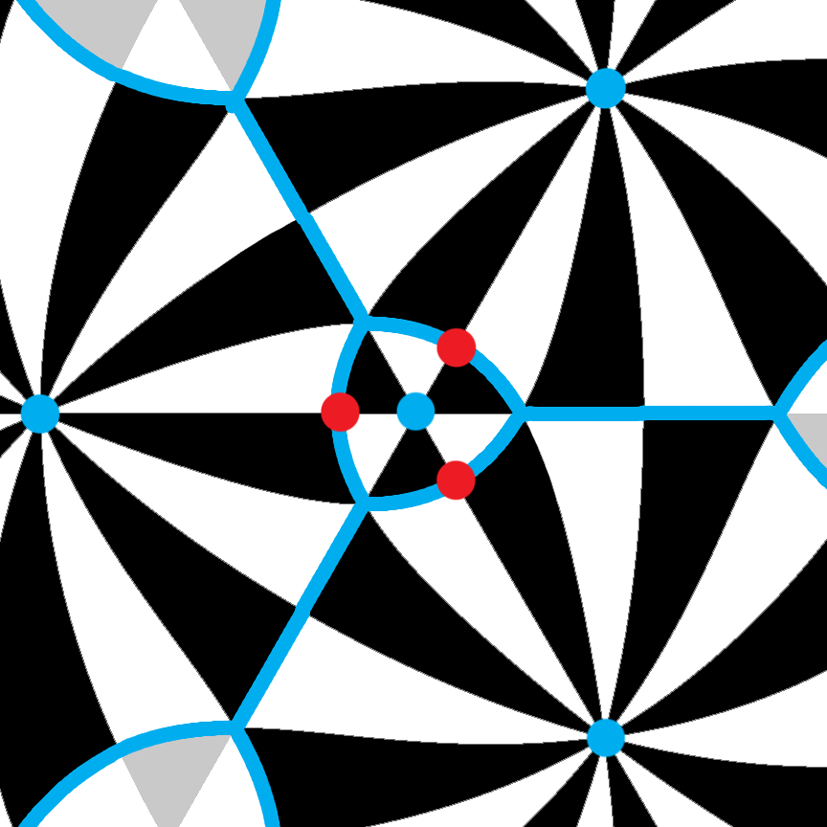}
    \end{overpic}
    \caption{The flower centered at $c_0$.}
    \label{fig:c0closeup}
  \end{subfigure}
  \phantom{XXX}
  \begin{subfigure}[b]{0.4\textwidth}
    \centering
    \begin{overpic}
      [width=5.5cm,
      tics=10]{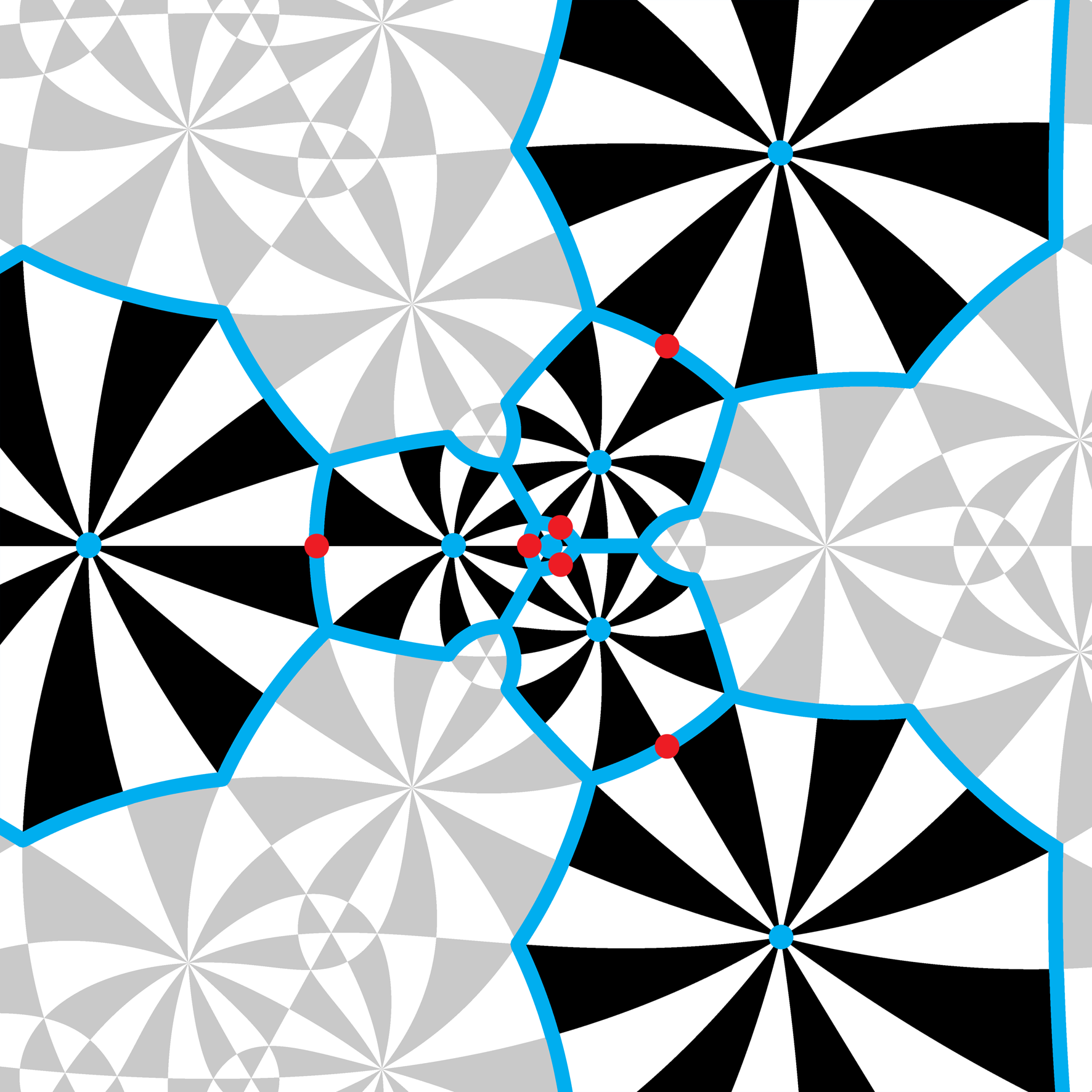}
    \end{overpic}
    \caption{The $b$-flowers around $c_0$.}
    \label{fig:c0flowers}
  \end{subfigure}
  \caption{Flowers around $c_0$.}
\end{figure}

\begin{proof}
  Suppose that $n\geq 2$ is an integer. Let us consider the critical point $c_0$, see
  Figure~\ref{fig:deff} and \eqref{eq:rami_f_triag_snow}. Note
  that $f_1^n(c_0) = 1$ and $\deg(f_1^n,c_0) = 3$. Thus, $c_0$ is an $n$-vertex of type $b$ and the degree of
  $W^n(c_0)$ is $3$. Figure~\ref{fig:c0closeup} 
  shows such a flower $W^n(c_0)$. 

  Note that the $0$-edge $[1,\infty]$ has three $f_1^2$-preimages  at
  $c_0$. More precisely, there are three analytic closed arcs
  $A_1,A_2,A_3$, such that they intersect (pairwise) precisely in $c_0$ and
  $f_1^2\colon A_j \to [1,\infty]$ is a homeomorphism for
  $j=1,2,3$. Denote by $A_j'$ the arc $A_j$ with the endpoint
  different from $c_0$ removed.  
  
  Let $j\in\{1,2,3\}$. It follows that on each arc $A_j'$ the combinatorial
  picture of level $n$ is the same as on $[1,\infty)$ for level $n-2$ (the latter one is described by
  Lemma~\ref{lem:flowers_1_infty}). The only difference is that the degree of $W^n(c_0)$ is 3, while the degree of $W^n(1)$ is 1. In particular, we can associate to each $n$-vertex $v\in A_j'$ of type $b$ a white $n$-tile $X(v) \subset W^n(v)$, so that $X(v)$ is mapped by $ab^k$ to $X(v')$, for $k=4,12,20$. Here, $v'$ is the next
  $n$-vertex of type $b$ on $A_j$ that follows $v$, when traversing $A_j$
  starting from $c_0$. \hide{This is true here even for the endpoint of $A_j$, corresponding to $\infty$ (unlike in the previous lemma).} Figure~\ref{fig:c0flowers} shows the
  $b$-flowers on the arcs $A_1,A_2,A_3$ for sufficiently large $n$
  (Figure~\ref{fig:c0closeup} is a close-up of
  Figure~\ref{fig:c0flowers}). For the convenience of the reader
  we also show a subgraph of the $n$-level Schreier graph in
  Figure~\ref{fig:Schreier_c0}, which may be viewed as a
  schematic version of Figure~\ref{fig:c0flowers}.

\begin{figure}
  \centering
  \begin{overpic}
    [width=7.5cm, tics=10,
    ]{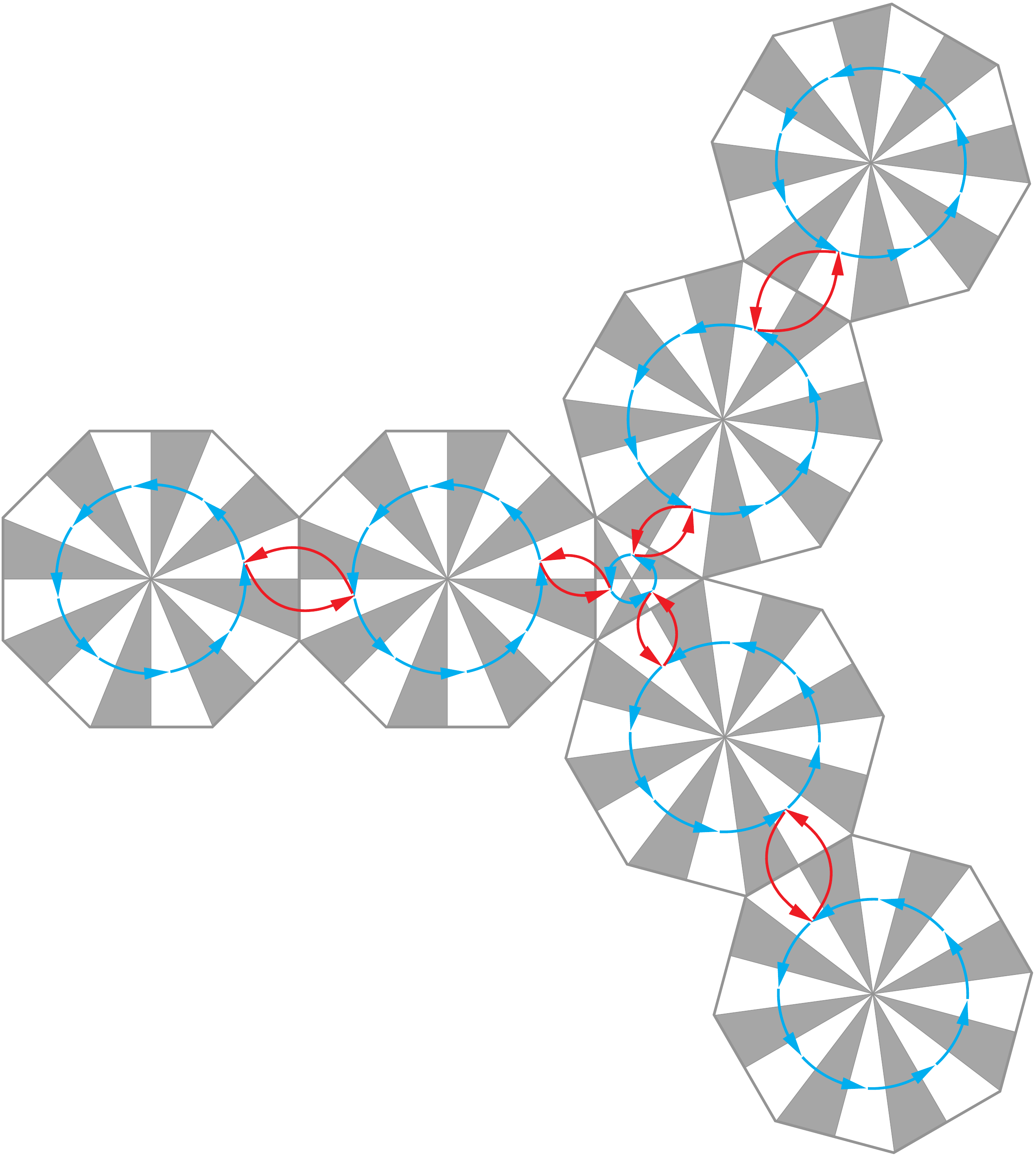}
    \put(-5, 48.7){$\cdots$}
    \put(82.7,-1.7){\rotatebox{120}{$\cdots$}}
    \put(82.7,97.5){\rotatebox{60}{$\cdots$}}
  \end{overpic}
  \caption{A subgraph of the Schreier graph at $c_0$.}
  \label{fig:Schreier_c0}
\end{figure}

\longhide{ 
  In particular, for each $n$-vertex $v\in A_j'$ of type $b$,
  there is a white $n$-tile $X(v) \subset W^n(v)$ that is mapped
  by $ab^k$ to $X(v')$, for $k=4,12,20$. Here, $v'$ is the next
  $n$-vertex of type $b$ on $A_j$ that follows $v$, when traversing $A_j$
  starting from $c_0$. Figure~\ref{fig:c0flowers} shows the
  $b$-flowers on the arcs $A_1,A_2,A_3$
  (Figure~\ref{fig:c0closeup} is a closeup of
  Figure~\ref{fig:c0flowers}). 
}

  \smallskip
  Consider now two distinct words $w_1$ and $w_2$ of the form
  \eqref{eq:words_free_semi}. By multiplying from the left with
  the inverse of the common initial word, it is enough to assume
  that 
  \begin{equation*}
    w_1 = b^{k_1} a b^{k_2} \dots a b^{k_{N}},
    \quad 
    w_2 = b^{m_1} a b^{m_2} \dots a b^{m_{N}},
  \end{equation*}
  where $k_1,\dots, k_{N}, m_1,\dots , m_{N}\in\{4,12,20\}$, and
  $k_1 \neq m_1$. Fix a sufficiently large $n$ (in fact it will be enough to
  demand that $2^{n-3} > N$). Fix one white $n$-tile $X$ in
  $W^n(c_0)$. Let us apply the two words $w_1$ and $w_2$ to $X$. Note that $b^{k_1}$ and $b^{m_1}$ map this tile to
  distinct $n$-tiles $X_1$ and $X_2$ in $W^n(c_0)$, since $8$ is not a multiple of $3$. We conclude that the remaining subwords $a b^{k_2} \dots a
  b^{k_N}$ and $a b^{m_2} \dots a b^{m_N}$ of $w_1$ and $w_2$ shift these
  $n$-tiles along two distinct arcs among
  $A_1,A_2,A_3$. Thus $w_1$ and $w_2$ map $X$ to distinct
  $n$-tiles.  Consequently, $w_1$ and $w_2$ are distinct elements in
  $\img(f_1)$.

  The previous argument can be applied to show that distinct words
  of the form \eqref{eq:words_free_semi} are distinct in $\img(f_1)$
  even if they have different length. However, we note that there is a special case when after cancellation of the common initial part of the two given words we are left with $w_1=1$ and $w_2\neq 1$. In this case $w_1$ fixes the $n$-tile $X$ selected above, while $w_2$ shifts it along one of the arcs $A_1,A_2,A_3$. Thus $w_1$ and $w_2$ are distinct elements in
  $\img(f_1)$.  Hence we have proved that the elements $ab^4, ab^{12}, ab^{20}$ generate a free
  semigroup in $\img(f_1)$. 
\end{proof}

From the previous lemma it follows immediately that $\img(f_1)$ is
of exponential growth. This proves
Theorem~\ref{thm:f1_exp_growth}.






\section{A criterion for exponential growth}
\label{sec:crit-expon-growth}

Here we analyze the essential ingredients used in
Section~\ref{sec:exponential-growth-r} to prove exponential
growth of the iterated monodromy group of the map $f_1$. This will
result in a somewhat general sufficient condition for the $\img$ of a
Thurston map to be of exponential growth. However, we point out that this criterion is far from being necessary. Moreover, the imposed conditions on the Thurston map can be further relaxed.

Let $g\colon S^2\to S^2$ be a Thurston map (see
Definition~\ref{def:Thurston-map}). We fix a Jordan curve
$\CC\subset S^2$ with $\post(g) \subset \CC$. As in Section \ref{sec:tiles-flow-iter}, the
$0$-edges are the closed arcs into which $\post(g)$ divides
$\CC$. For the map $f_1$ we considered the $0$-edge $[1,\infty]$
which was $f_1$-invariant. So in general we demand that
\begin{align}
  \label{eq:E_ginv}
  \tag{a}
  &\text{ there is a $g$-invariant $0$-edge $E$ with endpoints
    $p,q\in \post(g)$}
  \\
  \notag
  &\text{ such that $g|E\colon E\to E$ is not a homeomorphism.}
\end{align}

Put differently, $g(E)= E$ and $d_E\coloneqq \deg(g|E)\geq 2$. 
 It follows that the $0$-edge $E$ is
subdivided into $d_E^n$ $n$-edges by $n$-vertices for each $n\in \N_0$.

\hide{Recall that this means that $g(E)=E$.} Further, we assume that
\begin{equation}
  \label{eq:gpp}
  \tag{b}
  g(p)=p.
\end{equation}
Note that from $g(E)=E$ it follows that
$g(\{p,q\})\subset \{p,q\}$, which means that at least one of the
points $p,q$ is a fixed point of $g^2$. So the main purpose of
condition \eqref{eq:gpp} is to fix the notation, since $p$ and
$q$ will play somewhat different roles.

Let $v$ be an $n$-vertex on $E$ for some $n\in \N_0$, meaning that $v\in g^{-n}(\post(g))\cap E$. Since $E$ is $g$-invariant, it
follows that either $g^n(v)=p$ or $g^n(v)=q$. In the first case we say that the $n$-vertex $v$
is of type $p$, otherwise we say that it is of
type $q$. We again note that $v$ is an $(n+1)$-vertex as well, and
as such may be of different type. Recall that each $n$-edge $e$ in $E$ is mapped by
$g^n$ homeomorphically to $E$. If follows that one endpoint of $e$ is an
$n$-vertex of type $p$ and the other one is of type $q$. Put
differently, the types of the $n$-vertices on $E$ alternate.

\hide{For any $n\in \N_0$, the $0$-edge $E$ is
subdivided into $n$-edges. For any $n\in \N_0$ let
$v_1,\dots,v_N$ be the $n$-vertices on $E$, meaning the points in
the set $g^{-n}(\post(g))\cap E$. Since $E$ is $g$-invariant, it
follows that any such $n$-vertex $v_j$ satisfies $g^n(v_j)=p$ or
$g^n(v_j)=q$. In the first case we say that the $n$-vertex $v_j$
is of type $p$, otherwise we say that the $n$-vertex $v_j$ is of
type $q$. Note that $v_j$ will be an $(n+1)$-vertex as well, and
as such may be of different type. 
Recall that the $n$-vertices divide $E$ into the (closed)
$n$-edges contained in $E$. Each such $n$-edge $e$ is mapped by
$g^n$ homeomorphically to $E$. If follows that $e$ has one
$n$-vertex as its endpoint of type $p$ and one of type $q$. Put
differently, the types of $n$-vertices on $E$ alternates. 
}

The second essential ingredient for the example $f_1$ was that
for all $n$-vertices $v\in (1,\infty)$ of a fixed type the degrees
of the flowers $W^n(v)$ were the same. In general we demand that there are
constants $k_p, k_q\in \N$, such that for each $n\in\N$ every $n$-vertex
$v\in E\setminus\{p,q\}$ of type $p$, and every $n$-vertex
$w\in E\setminus\{p,q\}$ of type $q$, satisfies
\begin{equation}
  \label{eq:deggv}
  \tag{c}
  \deg(g^n,v) = k_p \text{ and } \deg(g^n, w) = k_q.
\end{equation}

\smallskip

There is an elementary way to check that condition
\eqref{eq:deggv} is true. Let $V_p\coloneqq g^{-1}(p)\cap (E
\setminus \{p,q\})$ and $V_q\coloneqq g^{-1}(q)\cap (E\setminus
\{p,q\})$. These are the $1$-vertices on $E\setminus\{p,q\}$ of
type $p$ and $q$ respectively. We first note that
(\ref{eq:deggv}) implies that there are constants $k_p,k_q\in \N$
with 
\begin{equation}
  \label{eq:c1}
  \tag{c1}
  k_p= \deg(g,v) \text{ and } k_q=\deg(g,w)
\end{equation}
for all $v\in V_p$ and $w\in V_q$.  

Let $v\in
E\setminus\{p,q\}$ be an $n$-vertex
of type $p$, then by \eqref{eq:gpp} it is an $(n+1)$-vertex of type $p$ as
well. Assuming (\ref{eq:deggv}) it follows that
\begin{equation*}
  k_p= \deg(g^{n+1},v) = \deg(g^n,v)\deg(g,p) = k_p \deg(g,p).
\end{equation*}
Thus (\ref{eq:deggv}) implies
\begin{equation}
  \label{eq:c2}
  \tag{c2}
  \deg(g,p)=1.
\end{equation}

The other $0$-vertex $q$ satisfies either $g(q)=q$ or $g(q)=p$.  
In the first case we obtain from an exactly analogous
argument that (\ref{eq:deggv}) implies
\begin{equation}
  \label{eq:c3}
  \tag{c3}
  \deg(g,q)=1  \quad\text{ in the case }\quad g(q)=q.
\end{equation}
Furthermore, in this case the three conditions (\ref{eq:c1}),
(\ref{eq:c2}), (\ref{eq:c3}) imply (\ref{eq:deggv}). 
Indeed, for any $n$-vertex $v\in E\setminus\{p,q\}$ of
type $p$ the sequence $v,
g(v),\dots,g^n(v)=p$ contains exactly one critical point of $g$,
which is contained in $V_p$, and thus has local degree
$k_p$. Hence 
$\deg(g^n,v)= k_p$. Similarly, $\deg(g^n,w)= k_q$ for each
$n$-vertex $w\in E\setminus\{p,q\}$ of type $q$. 

Let us now consider the second case $g(q)=p$. Suppose $w\in
V_q$. Then $g^2(w) = p$. From (\ref{eq:deggv}) it follows that
$k_p=\deg(g^2,w)=\deg(g,w)\deg(g,q)= k_q \deg(g,q)$. Thus we
obtain the condition
\begin{equation}
  \label{eq:c4}
  \tag{c4}
  k_p= k_q \deg(g,q)
  \quad \text{ in the case }\quad
  g(q)=p.
\end{equation}
Conversely, in this case the conditions (\ref{eq:c1}),
(\ref{eq:c2}), (\ref{eq:c4}) imply that for any $n$-vertex
$v\in E\setminus\{p,q\}$ of type $p$, the sequence
$v,g(v), \dots,g^n(v)=p$ contains either exactly one point in
$V_p$ or the point $q$ and exactly one point in $V_q$. In both situations
$\deg(g^n,v)=k_p$. The argument that $\deg(g^n,w)=k_q$ for each
$n$-vertex $w\in E\setminus\{p,q\}$ of type $q$ is the same as in
the first case.

In conclusion, we have seen that
\begin{equation*}
  (\ref{eq:c1}),~(\ref{eq:c2}),~(\ref{eq:c3}),~(\ref{eq:c4}) 
  \Leftrightarrow
  (\ref{eq:deggv}). 
\end{equation*}

In practice we often consider the \emph{ramification portrait of
  $g$ restricted to $E$}. Let
$v\in V_p\cup V_q\subset E\setminus \{p,q\}$. Since $v$ is
incident to two $1$-edges contained in $E$, both of which are
mapped to $E$, it follows that $v$ is a critical point of $g$. Thus the
ramification portrait of $g$ does indeed contain all $1$-vertices
in $E$, and we may restrict it to the set of these vertices. We
note that for a (critical) point $v\in V_p\cup V_q\cup\{p,q\}$ of degree
$d_v=\deg(g,v)$ we label the directed edge from $v$ to $g(v)$ by
``$d_v:1$''. Note that $d_v$ differs in general
from $\deg(g|E,v)$ (which is always $2$ for $v\in V_p\cup V_q$).

For example, the
ramification portrait of $f_1$ restricted to $[1,\infty]$ is
\begin{equation*}
  \xymatrix{
    a_0 \ar[r]^{2:1} & \infty \ar[r]^{4:1} & 1\ar@(ru,lu)[]\rlap{.}
  }
\end{equation*}
Given the restricted ramification portrait of a $g$-invariant
$0$-edge, it is immediate to verify that conditions
(\ref{eq:c1}),~(\ref{eq:c2}),~(\ref{eq:c3})~(\ref{eq:c4}) (as
well as (\ref{eq:gpp})) are satisfied.

\smallskip
Condition (\ref{eq:deggv}) is not yet sufficient to ensure that a
suitable word (which was $ab^4$ for the example $f_1$) acts by
``shifting white $n$-tiles along $E$''. 
Let $v\in E\setminus\{p,q\}$ be an $n$-vertex. Then there are $n$-edges $e,e'\subset E$ that
intersect in $v$. Then any $n$-tile in the flower $W^n(v)$ of level $n$ centered at $v$ is
either contained in the sector between $e$ and $e'$ or in the
sector between $e'$ and $e$. We refer to these two sectors as the
\emph{sectors into which $E$ divides $W^n(v)$}. 
In this setting we demand that
\begin{align}
  \label{eq:bw_above_below}
  \tag{d}
  &\text{for each $n$-vertex $v\in E\setminus\{p,q\}$ the
    two sectors into which $E$ } 
  \\
  \notag
  &\text{divides $W^n(v)$ contain the \emph{same number of}
    $n$-tiles.}  
\end{align}

\longhide{
\begin{align}
  \label{eq:bw_above_below}
  \tag{d}
  &\text{for any $n$-vertex $v\in E\setminus\{p,q\}$ each of the
    two sectors into} 
  \\
  \notag
  &\text{which $E$ divides $W^n(v)$ contains an \emph{equal number of}
    $n$-tiles.}  
\end{align}
}


Let $d_v = \deg(g^n,v)$. Then there are $2d_v$ 
$n$-edges that contain $v$. Let $e_0, e_1,\dots , e_{2d_v-1}$ be
these $n$-edges labeled cyclically around $v$ so that $e=e_0\subset E$. Then
\eqref{eq:bw_above_below} is equivalent to the requirement that
$e'=e_{d_v}\subset E$.

Since $E$ is invariant, $g^n(e_0)=g^n(e_{d_v})=E$. At the same time $g^n$ maps $e_j$ to $E$
if and only if $j$ is even. It follows that $d_v$ is even. Consequently, conditions \eqref{eq:E_ginv} and \eqref{eq:bw_above_below} imply that the two sectors into which $E$ divides $W^n(v)$ contain an (equal) even number of
$n$-tiles. 
Thus the numbers $k_p$ and $k_q$ from
condition~\eqref{eq:deggv} are even.

Let us color the $0$-tiles black and white. As in Section
\ref{sec:tiles-flow-iter}, the dynamics of $g$ defines a coloring
on the $n$-tiles respecting the above choice for all $n\in\N$.
The colors of the $n$-tiles containing $v$ (that is, the
$n$-tiles in the flower $W^n(v)$) alternate cyclically around
$v$. From the discussion above, it follows that conditions \eqref{eq:E_ginv} and \eqref{eq:bw_above_below} mean that the numbers of white
$n$-tiles in the two sectors into which $E$ divides $W^n(v)$  are
equal. 

\longhide{Let $d_v = \deg(g^n,v)$. Then there are $2d_v$ 
$n$-edges that contain $v$. Let $e=e_0, e_1,\dots , e_{2d_v-1}$ be
these $n$-edges, labeled cyclically around $v$. Then
\eqref{eq:bw_above_below} is equivalent to the requirement that
$e'= e_{d_v}$ (meaning the two sectors into which $E$ divides
$W^n(v)$ contain an equal number of $n$-edges).

Note that $g^n(e)=E$. Consequently $g^n$ maps $e_j$ to $E$
exactly when $j$ is even (and to the other $0$-edge containing
$g^n(v)\in \{p,q\}$ when $j$ is odd). It follows that
\eqref{eq:bw_above_below} implies that $d_v$ is
even. This means that the numbers $k_p$ and $k_q$ from
condition~\eqref{eq:deggv} are even. 

Furthermore the two sectors
into which $E$ divides $W^n(v)$ contain an (equal) even number of
$n$-tiles.

Let us color the $0$-tiles black and white. As in Section
\ref{sec:tiles-flow-iter}, the dynamics of $g$ defines a coloring
on the $n$-tiles respecting the above choice for all $n\in\N$.
The colors of the $n$-tiles containing $v$ (that is, the
$n$-tiles in the flower $W^n(v)$) alternate cyclically around
$v$. Since each of the two sectors between $e$ and $e'$ contains
an even number of $n$-tiles, it follows that condition \eqref{eq:bw_above_below} means that the numbers of white
$n$-tiles in the two sectors between $e$ and $e'$ in $W^n(v)$ are
equal. 
}


Note that \eqref{eq:bw_above_below} is automatically satisfied if $E\subset \widehat{\R}$ and $g$ is a real rational function, that is, $g(\RDach)\subset \RDach$.



As for condition \eqref{eq:deggv}, there is an equivalent
condition that only involves $1$-vertices, $1$-tiles, and 
$1$-flowers. More precisely, we demand that
\begin{align}
  \label{eq:d'}
  \tag{d'}
  &\text{for each $1$-vertex $v\in E\setminus\{p,q\}$ the
    two sectors into which $E$} 
  \\
  \notag
  &\text{divides $W^1(v)$ contain the \emph{same number of}
    $1$-tiles.}  
\end{align}

\longhide{
\begin{align}
  \label{eq:d'}
  \tag{d'}
  &\text{for any $1$-vertex $v\in E\setminus\{p,q\}$ each of the
    two sectors into} 
  \\
  \notag
  &\text{which $E$ divides $W^1(v)$ contains an \emph{equal number of}
    $1$-tiles.}  
\end{align}
}


\begin{lemma}
  \label{lem:div_sector_1v_nv}
  Let $E$ be a $g$-invariant $0$-edge with the endpoints $p$ and $q$ as in
  \eqref{eq:E_ginv}. Then condition \eqref{eq:bw_above_below} is
  equivalent to condition~\eqref{eq:d'}. 
\end{lemma}

\hide{\textsf{Warning: Here, we use the fact that $E$ is $g$-invariant, otherwise we also need to assume that the black tiles are symmetric!}}

\begin{proof}
  The implication \eqref{eq:bw_above_below} $\Rightarrow$
  \eqref{eq:d'} is trivial. 

  Let us show the other implication. First
  consider a $1$-vertex $v\in E\setminus\{p,q\}$; recall that
  $v$ then is an $n$-vertex as well for all $n\in\N$.

  Since the $0$-edge $E$ is $g$-invariant, $g(v)\in \{p,q\}$. Assume that $g(v)=p$. The
  other case is completely analogous and will not be treated 
  separately.  Let $E$ and $\widetilde{E}$ be the $0$-edges
  incident to $p$. Then the $1$-edges containing $v$ are
  alternatingly mapped to $E$ and $\widetilde{E}$ by $g$. Suppose that the $1$-edges around $v$ are
  $e_0, \widetilde{e}_0, e_1, \widetilde{e}_1, \dots, e_{d_v}=e_0$ in
  cyclic order. Here $d_v= \deg(g,v)$,  $g(e_j)=E$, and
  $g(\widetilde{e}_j) = \widetilde{E}$ for $j=0,\dots, d_{v}-1$. Let
  $e_0 \subset E$, then by conditions \eqref{eq:E_ginv} and \eqref{eq:d'}
  it follows that the other $1$-edge in $E$ containing $v$ is
  $e_{d_{v}/2}$.

  Fix an $n\in \N$ and view the $1$-vertex $v\in E\setminus\{p,q\}$ as an $n$-vertex. 
  Consider all the $(n-1)$-edges containing $g(v)=p$. Denote these $(n-1)$-edges by $K_0,\dots, K_{m-1}$ so that $K_0\subset E$, where $m=2\deg(g^{n-1},g(v))$. Then every $1$-edge $e_j$, defined above,
  contains an $n$-edge $k_{j,0}$ incident to $v$, such that $g(k_{j,0})=K_0$. Moreover, between $e_j$
  and $e_{j+1}$ there are exactly $(m-1)$ $n$-edges containing $v$. It follows that
 the two sectors in $W^{n}(v)$ between the two $n$-edges in $E$ containing $v$, namely $k_{0,0}$ and $k_{d_{v}/2,0}$, contain the same number of $n$-edges, proving the lemma in this case. 

  \smallskip
  Consider now an $n$-vertex $v\in E\setminus\{p,q\}$ that is not
  a $1$-vertex. Then there is a $j\leq n-1$ such that $w\coloneqq
  g^j(v)\in
  E\setminus\{p,q\}$ is a $1$-vertex. Since $w$ is not a
  postcritical point of $g$, it follows that $v$ is not a critical
  point of $g^j$. Thus there is a neighborhood of $v$ on which
  $g^j$ is a homeomorphism. Consequently, $g^j$ maps $n$-edges containing
  $v$ to $(n-j)$-edges containing $w$ bijectively, and at the same time $E$ to $E$. The
  desired statement for $v$ follows now from the corresponding statement for
  the $1$-vertex $w$ already proved above. 
\end{proof}

\smallskip

Finally, we assume the following condition.
\begin{align}
  \label{eq:prep_rel_prime}
  \tag{e}
  &\text{There is a point $c\in S^2$ with $g^k(c)=p$ for a $k\in
    \N$}
  \\
  \notag
  &\text{such that $d_c\coloneqq\deg(g^k, c)$ 
  is not a divisor of $k_p$.}
\end{align}

As for the map $f_1$ we define the elements $a$ and
$b$ of $\img(g)$ as being represented by loops around $q$ and
$p$, respectively. More precisely, we fix a basepoint $t$ in the
interior of the white $0$-tile denoted by $X^0_\wt$. Then the loop $a$ is represented
by connecting $t$ in $X^0_\wt$ to a small circle around $q$. Similarly, the loop
$b$ is represented by connecting $t$ in $X^0_\wt$ to a small circle around $p$.


Assuming conditions \eqref{eq:E_ginv}--\eqref{eq:bw_above_below},
one shows exactly analogous to Corollary~\ref{cor:ab4_inf_order} 
that $a^{k_q/2}b^{k_p/2}$ is of infinite order in $\img(g)$. If we also assume condition \eqref{eq:prep_rel_prime}, it follows 
exactly as in Lemma~\ref{lem:free_semi_gp} that 
$a^{k_q/2}b^{k_p/2}$ and $a^{k_q/2}b^{k_p/2}b^{k_p}$ generate a
free semigroup in $\img(g)$. This means we have proved the
following.

\begin{theorem}
  \label{thm:inf_order_gen}
  Let $g\colon S^2\to S^2$ be a Thurston map, and $\CC\subset
  S^2$ be a Jordan curve with $\post(g) \subset \CC$ that
  satisfies conditions
  \eqref{eq:E_ginv}--\eqref{eq:bw_above_below}. Then $\img(g)$ contains an element of infinite order.  
\end{theorem}

\begin{theorem}
  \label{thm:exp_growth_gen}
  Let $g\colon S^2\to S^2$ be a Thurston map, and $\CC\subset
  S^2$ be a Jordan curve with $\post(g) \subset \CC$ that
  satisfies conditions
  \eqref{eq:E_ginv}--\eqref{eq:prep_rel_prime}. Then $\img(g)$ is
  of exponential growth.  
\end{theorem}

Let us provide some remarks on how our conditions may be relaxed.
\begin{remarks}
\label{rem:vary_exp_growth}
\phantom{X}

\mbox{}   

 (1) For any Thurston map $g\colon S^2\to S^2$, $\img(g)$ is
  isomorphic to $\img(g^n)$ for any $n\in \N$. Thus to prove
  exponential growth of $\img(g)$ it is enough to check that our
  conditions are satisfied for some iterate $g^n$.


   (2) Two Thurston maps $f\colon S^2\to S^2$ and $g\colon S^2\to S^2$
  that are Thurston equivalent have isomorphic iterated monodromy
  groups \cite[Corollary 6.5.3]{Nekra}. Thus our conditions, in particular \eqref{eq:E_ginv},
  need only be satisfied up to isotopy rel.\ $\post(g)$. 


    (3) For simplicity, let us denote $x:=a^{k_p/2}b^{k_q/2}$. By \eqref{eq:gpp} there exists a unique point $t_0\in g^{-1}(t)\cap W^1(p)$. Conditions \eqref{eq:E_ginv}--\eqref{eq:bw_above_below} imply that there is a sequence of distinct points $t_0, t_1, \dots, t_{n-1} \in g^{-1}(t)$, such that $t_{j+1 \bmod(n)} = t_j^x$, that is, the lift $x_j$ of $x$ starting at $t_j$ ends at $t_{j+1 \bmod (n)}$ for each $j=0,\dots,n-1$. In fact, $n=d_E=\deg(g|E)$. Let $\gamma$ be the closed curve in $\Sp\setminus\post(f)$ obtained by concatenation of the paths $x_j$, $j=0,\dots,n-1$. It follows that the curve $\gamma$ is homotopic to $x$ rel.\ $\post(g)$. Furthermore $\deg(g|\gamma)=n >1$. Loosely speaking, we can say that $x$ acts on the white $n$-tiles
  sharing an edge with $E$ by ``permuting them cyclically around $E$''. In this way, one can relax conditions \eqref{eq:E_ginv}, \eqref{eq:deggv}, and \eqref{eq:bw_above_below} by requiring existence of an element $x\in\img(g)$ with the properties mentioned above. This element $x$ will be automatically of infinite order in $\img(g)$. However, one will also need to adapt condition \eqref{eq:prep_rel_prime} in a certain way to conclude the exponential growth of $\img(g)$. The resulting criterion however is somewhat cumbersome and will not be formulated here.

\end{remarks}




\section{Sierpi\'{n}ski carpet rational maps }
\label{sec:sierp-carp-rati}

In this section we present a family of postcritically-finite rational maps,
such that each map in the family has a Sierpi\'{n}ski carpet as its Julia set, and
iterated monodromy group of exponential growth. 
Here, we call a set $\SC\subset \CDach$ a \emph{Sierpi\'{n}ski carpet} if
and only if it is homeomorphic to the standard Sierpi\'{n}ski
carpet. By Whyburn's characterization, this is the case if and
only if $\SC$ is compact, connected, locally connected, has
topological dimension $1$, and no local cut-points. Equivalently,
there exists a sequence $\{D_j\}_{j\in \N}$ of Jordan domains in
$\CDach$ having pairwise disjoint closures, such that $\SC = \CDach\setminus (\bigcup_j D_j)$ has empty interior and
$\diam(D_j) \to 0$ as $j\to \infty$, see
\cite{Why}. These two characterizations show that a Sierpi\'{n}ski
carpet is a universal object and justify the study of rational
maps with Sierpi\'{n}ski carpet Julia set. The first example of
such a rational map is due to Milnor and Tan Lei, see \cite[Appendix]{MilTan_Sierp}.


The rational maps we will consider now have originally been
studied by Ha\"{i}ssinsky and Pilgrim in
\cite{HP_Sierp} (with one difference that we address
at the end of the section). We briefly review the construction, which is similar to 
the way the map $f_1$ from Section~\ref{sec:constr-f} was
defined. 

To keep the discussion elementary, we first consider
one concrete map from our family. Let $\Delta$ and $\Delta'$ be the two
polyhedral surfaces shown in Figure~\ref{fig:f_sierp}. Here, $\Delta$ is a pillow obtained by
gluing two copies of the unit square $[0,1]^2$, called faces of $\Delta$, together along
their boundaries. The polyhedral surfaces $\Delta'$ is constructed from
$2\times 11$ squares of side length $1/3$, called
\emph{$1$-squares} of $\Delta'$ 
(they correspond to $1$-tiles in
the terminology of Section~\ref{sec:tiles-flow-iter}). More
precisely we take two copies of
\begin{equation*}
  Y\coloneqq 
  [0,1]^2 ~
  \cup 
  ~ [0,\tfrac{1}{3}]\times [-\tfrac{1}{3},0]  ~
  \cup  ~
  ~ [-\tfrac{1}{3},0]\times [0,\tfrac{1}{3}] ~
  \subset 
  \R^2
\end{equation*}
and glue them together along their boundaries. We call each copy
of $Y\subset \Delta'$ a face of $\Delta'$. Clearly $\Delta$
and $\Delta'$ are homeomorphic to $S^2$. As polyhedral surfaces,
$\Delta$ and $\Delta'$ may both be naturally viewed as Riemann
surfaces. With this point of view, there are conformal maps 
\begin{equation*}
  \varphi\colon \Delta\to\CDach
  \text{ and }
  \varphi'\colon \Delta'\to \CDach. 
\end{equation*}
By symmetry, we may assume that $\varphi$ and $\varphi'$ map one
face of $\Delta$ and, respectively, $\Delta'$ to the upper half-plane
and the other face to the lower half-plane, so that the points marked
$-1,0,1,\infty$ in Figure~\ref{fig:f_sierp} are mapped to
$-1,0,1,\infty\in \widehat{\R}\subset \CDach$, respectively. In fact,
$\varphi$ and $\varphi'$ may be constructed explicitly by mapping
each face of $\Delta$ and, respectively, $\Delta'$ by a Riemann map to
the upper or lower half-plane, where the four marked points are mapped
as indicated.

The map $g\colon \Delta' \to \Delta$ is now given as follows. 
Each $1$-square of $\Delta'$ 
is mapped (conformally) by a similarity that scales by the factor $3$ to a
face of $\Delta$ as indicated in
Figure~\ref{fig:f_sierp}. We define the map
\begin{equation*}
  f\coloneqq 
  \varphi \circ g \circ (\varphi')^{-1}\colon \CDach \to \CDach.
\end{equation*}
Note that $f$ is holomorphic, hence a rational map. Consider the
points in $\Delta'$ where at least four $1$-squares of $\Delta'$
intersect. Their images under $\varphi'$ are exactly the critical
points of $f$. Furthermore, the points $-1,0,1,\infty\in\CDach$,
which are the images of the vertices of $\Delta$ under $\varphi$,
are exactly the postcritical points of $f$. So, $f$ is a rational
Thurston map with $\post(f)=\{-1,0,1,\infty\}$. Moreover, the
Julia set of $f$ is a Sierpi\'{n}ski carpet, see \cite{HP_Sierp}.

\begin{figure}
  \centering
  \begin{overpic}
    [width=12cm, tics=20,
    ]{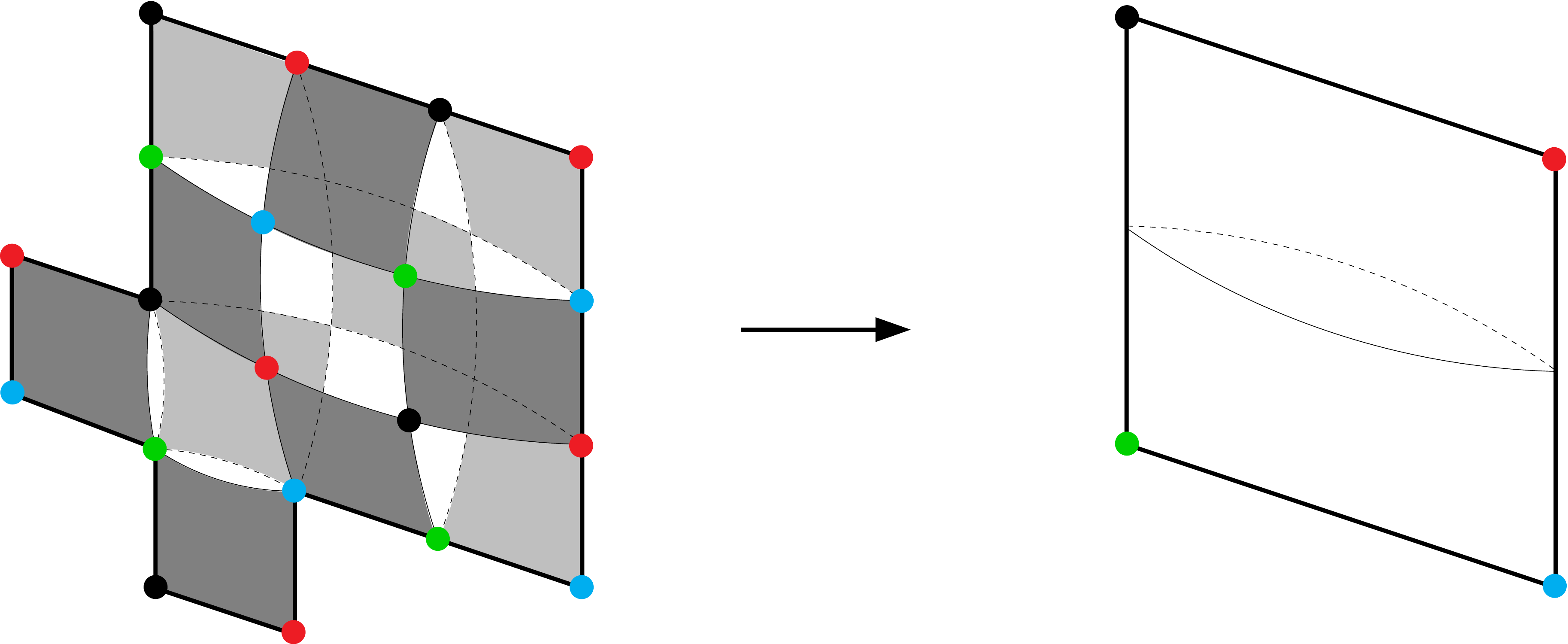}
    %
    \put(70,9.3){$0$}
    \put(67,39){$-1$}
    \put(100.2,2){$1$}
    \put(100.2,30){$\infty$}
    \put(51,21.5){$g$}
    %
    \put(7.5,9.7){$0$}
    \put(5,39){$-1$}
    \put(38,2){$1$}
    \put(38,30){$\infty$}
    \put(38.3,11.7){$a_0$}
    \put(38.3,20.5){$b_0$}
    \put(19.5,6.5){$c_0$}
    \put(23,38){$\Delta'$}
    \put(85,38){$\Delta$}
  \end{overpic}
  \caption{Construction of a Sierpi\'{n}ski carpet map.}
  \label{fig:f_sierp}
\end{figure}
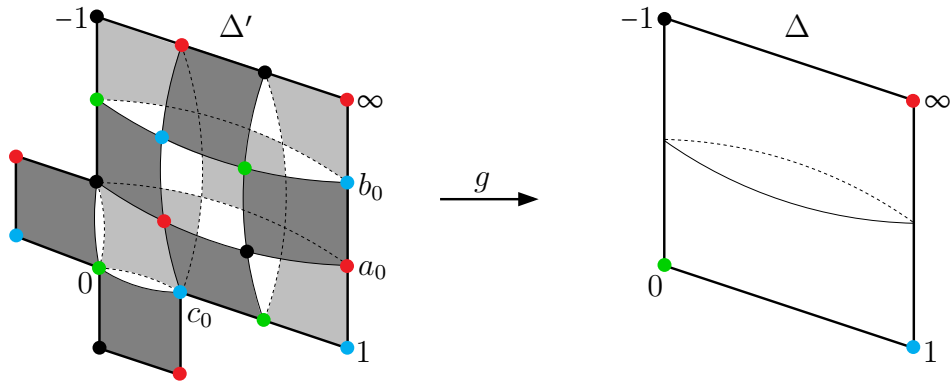


We mention in passing that we may think of $f$ as being
constructed via a \emph{two-tile subdivision rule} in the sense
of Cannon-Floyd-Parry (see \cite{CFKP} and
\cite[Chapter~12]{THEbook}). Alternatively, $f$ is constructed
from a Latt\`{e}s map by ``adding a flap'' or ``blowing up an
arc'' in the sense of \cite{PT}. 

Let $\CC\coloneqq \widehat{\R}$, $E\coloneqq [1,\infty]\subset
\CC$, $p\coloneqq1$, and $q\coloneqq \infty$. The ramification portrait of $f$
restricted to $E$ is 
\begin{equation*}
  \xymatrix{
    b_0 \ar[r]^{2:1} & 1\ar@(r,u)[] 
    & a_0 \ar[r]^{2:1} & \infty \ar@(r,u)[]\rlap{.}
  }
\end{equation*}
Again, by a slight abuse of notation, we denote the images of the points labeled by $a_0$, 
$b_0$, and $c_0$ in Figure~\ref{fig:f_sierp} under $\varphi'$ by $a_0$, $b_0$, and $c_0$, respectively.

Note that the critical point $c_0$ of $f$ satisfies $\deg(f,c_0)= 3$, see Figure~\ref{fig:f_sierp}. 
It is now
elementary to check that with the above choices conditions
\eqref{eq:E_ginv}--\eqref{eq:prep_rel_prime} from Section \ref{sec:crit-expon-growth} are satisfied. It
follows from Theorem~\ref{thm:exp_growth_gen} that $\img(f)$ is
of exponential growth. Thus we have proved Theorem~\ref{thm:fsierp_exp_growth}. 

We may vary the construction of the map $f$. In particular,
divide each of the two faces of the pillow $\Delta$ in
$n\times n$ squares of side length $1/n$, where $n\geq 3$ is an odd number, and add two diagonally symmetric flaps as in Figure~\ref{fig:f_sierp} to obtain a polyhedral surface $\Delta'$. Then repeat the above construction. Again the
resulting rational map $f$ has a Sierpi\'{n}ski carpet as its Julia set,
and $\img(f)$ is of exponential growth. The maps in
\cite{HP_Sierp} are constructed in the same fashion,
but there $n\geq 2$ is even. This results in a \emph{hyperbolic} rational map $f$, meaning
that every critical point of $f$ is contained in the Fatou set $\FC_f$. For each such map $f$ condition~\eqref{eq:deggv} is not satisfied. Consequently,
Theorem~\ref{thm:exp_growth_gen} does not apply. We do not know
if these hyperbolic maps have iterated monodromy groups of
exponential growth. However, the argument in \cite{HP_Sierp}, showing that
$f$ has a Sierpi\'{n}ski carpet Julia set, is equally valid
independently of whether $n$ is even or odd.


\section{A family of obstructed maps}
\label{sec:family-obstr-maps}

Here, the example from the previous section is slightly modified to
obtain an infinite family of obstructed Thurston maps with
iterated monodromy groups of exponential growth. 

The construction is very similar to the one in Section \ref{sec:sierp-carp-rati}. The only difference is that instead of ``adding two flaps'' to a Latt\`{e}s map we add only one; see Figure~\ref{fig:obtrs_map} and
compare with Figure~\ref{fig:f_sierp}. As before, we obtain a map
$g\colon \Delta'\to \Delta$ and define homeomorphisms
$\varphi\colon \Delta \to \CDach$ and
$\varphi'\colon \Delta'\to \CDach$ that are normalized as in the
previous section. That is, the top and bottom faces of $\Delta$ (and $\Delta'$)
are mapped to the upper and lower half-planes in $\CDach$, respectively, so that the points
labeled by $-1,0,1,\infty$ in $\Delta$ (and $\Delta'$) are mapped to
$-1,0,1,\infty\in \CDach$, respectively. 
We obtain a Thurston
map
$f\coloneqq \varphi\circ g \circ (\varphi')^{-1} \colon \CDach
\to \CDach$
such that $\img(f)$ is of exponential growth (using
Theorem~\ref{thm:exp_growth_gen} in the exact same form as in the
previous section). 


However, we want to point out a crucial difference from the situation in Section \ref{sec:sierp-carp-rati}. 
Namely, here each face of
$\Delta'$ is not symmetric with respect to the (diagonal) geodesic
joining the vertices labeled by $0$ and $\infty$ (in the respective face).
This means we
cannot choose the map $\varphi'\colon \Delta'\to \CDach$ to be conformal with
our normalization. Consequently, the constructed map $f$ is
not rational. 



With a slight abuse of notation consider now an arbitrary Thurston map $f\colon\Sp\to\Sp$. Thurston gave a criterion when the map $f$
is Thurston equivalent to a rational map. We present it only in the
case when $f$ has $4$ postcritical points and a hyperbolic orbifold. We refer to
\cite{DH_Th_char} for the general statement, as well as the
terminology. Let $\gamma\subset \Sp\setminus \post(f)$ be a Jordan curve
that is \emph{non-peripheral}, meaning that each component of
$\Sp \setminus \gamma$ contains at least $2$ postcritical
points (since $\#\post(f)=4$ this means that each component
contains exactly $2$ postcritical points). Let $\gamma_1,\dots ,\gamma_k$ be the
components of $f^{-1}(\gamma)$ that are non-peripheral. The curve
$\gamma$ is called \emph{invariant} if one (or, equivalently in our case, each) of the curves
$\gamma_j$, $j=1,\dots, k$, is homotopic to $\gamma$ rel.\ $\post(f)$. Denote by $d_j$ the degree of the restriction
$f\colon \gamma_j \to \gamma$ for $j=1,\dots,k$. Assume that $\gamma$ is invariant and define
\begin{equation}\label{eq:Thurst_coeff}
  \lambda_f(\gamma)\coloneqq \sum_{j=1}^{k} \frac{1}{d_j}.
\end{equation}
Then the curve $\gamma$ is called a \emph{Thurston obstruction}
if $\lambda_f(\gamma) \geq 1$. Thurston's theorem now says that
$f$ is Thurston equivalent to a rational map if and only if $f$
has no Thurston obstruction. Otherwise $f$ is called an
\emph{obstructed} Thurston map.

\begin{figure}
  \centering
  \begin{overpic}
    [width=11cm, tics=10,
    ]{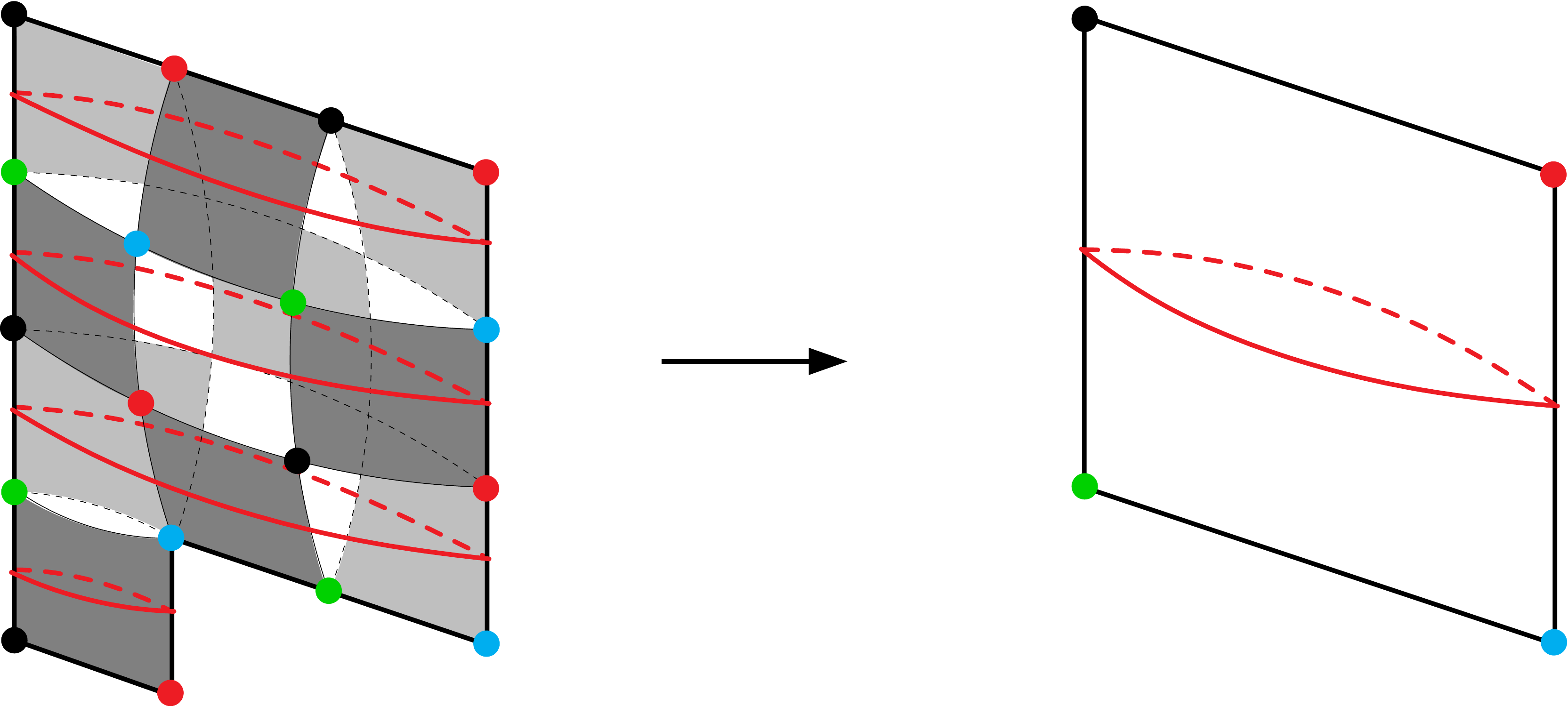}
    \put(68,10){$0$}
    \put(64,42.6){$-1$}
    \put(100.4,2.7){$1$}
    \put(100.4,33){$\infty$}
    \put(100,18){$\gamma$}
    \put(47,24){$g$}
    %
    \put(32,29){$\gamma_1$}
    \put(32,18){$\gamma_2$}
    \put(32,8){$\gamma_3$}
    \put(12,5){$\gamma_4$}
    \put(-2.5,12){$0$}
    \put(-4.3,43){$-1$}
    \put(32,2){$1$}
    \put(30,35.3){$\infty$}
    %
    \put(17,40){$\Delta'$}
    \put(84,40){$\Delta$}    
  \end{overpic}
  \caption{Construction of an obstructed map.}
  \label{fig:obtrs_map}
\end{figure}
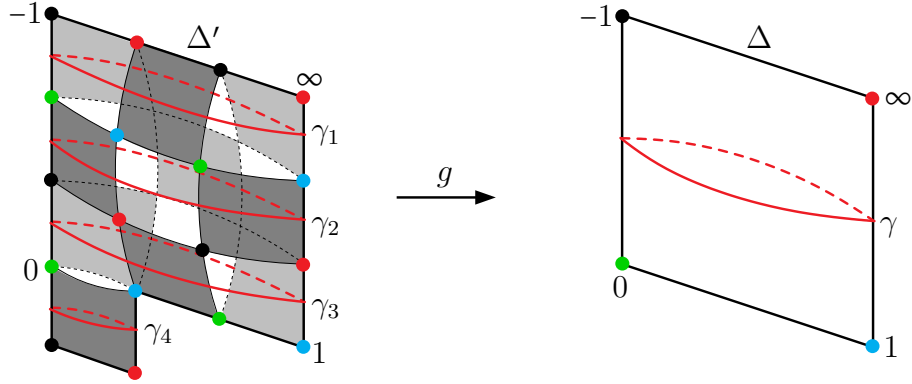

It follows that the map $f:\CDach\to\CDach$ constructed in this section
has a Thurston obstruction $\gamma$ as shown in
Figure~\ref{fig:obtrs_map}. Here, we choose to draw $\gamma$ in
$\Delta$ and its preimages $\gamma_1$, \dots, $\gamma_4$ in $\Delta'$ for simplicity. In
reality, we consider the images of these curves under $\varphi$ and $\varphi'$ in
$\CDach$, respectively. Then $\lambda_f(\gamma)=1$ (note that the
component $\gamma_4$ of $f^{-1}(\gamma)$ is peripheral, thus it does
not contribute to the sum \eqref{eq:Thurst_coeff}). Consequently, $\gamma$ is a Thurston
obstruction. Thus $f$ is not Thurston equivalent to a rational
map by Thurston's theorem, meaning it is obstructed. 
 We have proved Theorem~\ref{thm:f_obtructed_exp_growth}.

\section{A non-renormalizable polynomial $P$ with $\img$ of exponential growth}
\label{sec:non-renorm-polyn}

Here, we present an example of a postcritically-finite,
non-renormaliz\-able polynomial $P\colon \CDach \to \CDach$ with dendrite Julia set and
iterated monodromy group of exponential growth. This polynomial serves as an example in Theorem~\ref{thm:P_not_renorm_exp_growth}. In fact, $P$ is
given by
\begin{equation*}
  \label{eq:defP}
  P(z) = \frac{2}{27} (z^2+3)^3(z^2-1) +1 
  =
  \frac{2}{27}z^8 + \frac{16}{27}z^6 + \frac{4}{3} z^4 -1. 
\end{equation*}
From these two expressions we immediately see that $P$ has the
(finite) critical points $\pm \sqrt{3}i$ of local degree $3$, and
$0$ of local degree $4$.
Furthermore, they are mapped as follows
\begin{equation*}
  \label{eq:rami_port_P}
  \xymatrix{
    \pm \sqrt{3} i \ar[r]^{3:1} 
    & 
    1\ar@(ru,lu)[]
    & 
    -1 \ar[l] 
    &
    0 \ar[l]_{4:1}.
  }
\end{equation*}
Thus $\post(P)=\{-1,1,\infty\}$.
Let $p\coloneqq1$, $q\coloneqq-1$, $\CC\coloneqq\widehat{\R}$, $E\coloneqq[-1,1]$, and
$c\coloneqq\sqrt{3}i$. With these choices conditions
\eqref{eq:E_ginv}--\eqref{eq:prep_rel_prime} are satisfied. From
Theorem~\ref{thm:exp_growth_gen} it follows that $\img(P)$ has
exponential growth. Hence $P$ satisfies property \ref{cond:exp_growth} of Theorem \ref{thm:P_not_renorm_exp_growth}.

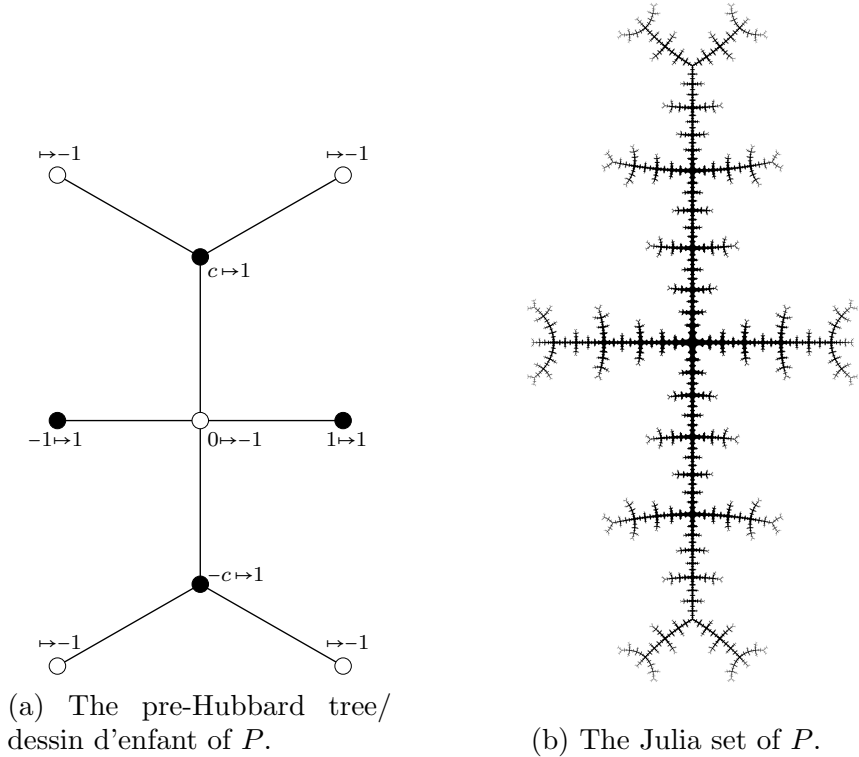
\begin{figure}
  \centering
  \begin{subfigure}[b]{0.4\textwidth}
    \centering
    \begin{overpic}
      [width=4cm, tics=10,
      ]{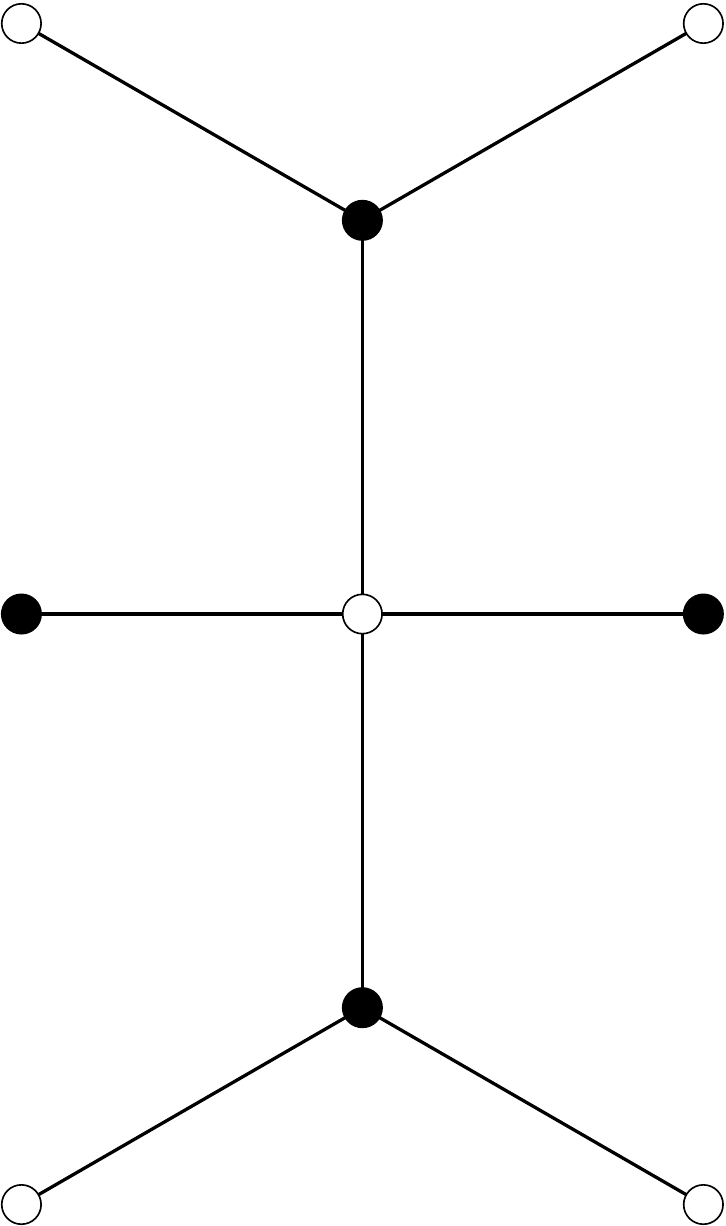}
      \put(-4,45){$\scriptstyle{-1\mapsto 1}$}
      \put(54,45){$\scriptstyle{1\mapsto 1}$}
      \put(31,45){$\scriptstyle{0\mapsto -1}$}
      \put(31,78){$\scriptstyle{c\,\mapsto 1}$}
      \put(31,18.5){$\scriptstyle{-c\,\mapsto 1}$}
      \put(-2,101){$\scriptstyle{\mapsto -1}$}
      \put(54,101){$\scriptstyle{\mapsto -1}$}
      \put(-2,5){$\scriptstyle{\mapsto -1}$}
      \put(54,5){$\scriptstyle{\mapsto -1}$}
    \end{overpic}
    \caption{The pre-Hubbard tree/ dessin d'enfant of~$P$.}
    \label{fig:P_dessin}
  \end{subfigure}
  \phantom{XXX}
  \begin{subfigure}[b]{0.4\textwidth}
    \centering
    \begin{overpic}
      [width=5.5cm,
      tics=10]{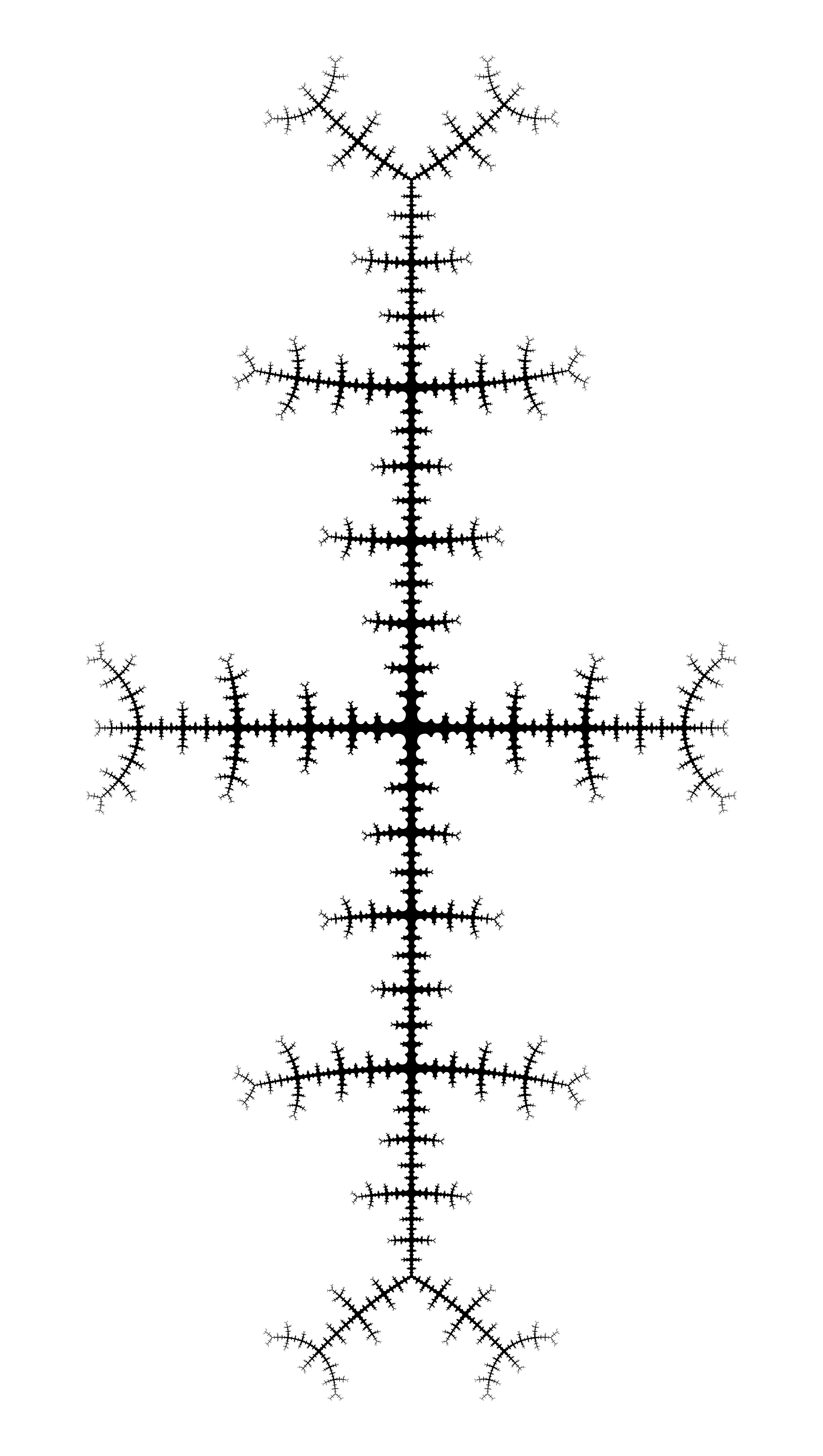}
    \end{overpic}
    \caption{The Julia set of $P$.}
    \label{fig:P_Julia}
  \end{subfigure}
  \caption{Visualizing $P$.}
\end{figure}


Clearly $P$ is postcritically-finite and has no (finite) periodic
critical points. It follows that the Julia set $\JC$ of
$P$, shown in Figure~\ref{fig:P_Julia}, is a \emph{dendrite} (that is, a compact, connected, locally connected set
with empty interior which does not separate the plane), see
\cite[\S 11.2]{BearIter}.
Consequently, $P$ satisfies property
\ref{cond:J_dendrite} of
Theorem~\ref{thm:P_not_renorm_exp_growth}.  
 
Since $\JC$ is a dendrite, the \emph{Hubbard tree} of $P$ may be defined as the smallest
continuum in $\JC$ containing all (finite) postcritical
points (see \cite{DH_Orsay}). From $P([-1,1]) = [-1,1]$ we
conclude that $[-1,1]$ is contained in the Julia set
$\mathcal{J}$ (because each orbit in $[-1,1]$ is bounded). Hence,
$H=[-1,1]$ is the Hubbard tree of $P$ and the Julia set $\JC$ coincides with the closure
\[\overline{\bigcup_{n\geq 0}P^{-n}(H)}.\]

Let us color the point $1$
black and $-1$ white. The preimage $H'\coloneqq P^{-1}(H)$ of the
Hubbard tree $H$ is schematically shown in
Figure~\ref{fig:P_dessin}. Here, we color the points in
$P^{-1}(1)$ black and the points in $P^{-1}(-1)$ white, and
indicate how they are mapped for convenience (even though this
information is already contained in the coloring). We also label
the critical and postcritical points.

Let us recall that a point $x$ of a dendrite $X$ is called a
\emph{leaf} of $X$ if $X\setminus\{x\}$ is connected. Clearly,
the points $1$ and $-1$ are leaves of the Hubbard tree $H$ and
the pre-Hubbard tree $H'$. It follows that each set $P^{-n}(H)$, $n\in\N_0$, is a planar tree and its set of leaves contains points $1$, $-1$.
Consequently,
the postcritical points $1$ and $-1$ are leaves of the dendrite
Julia set $\JC$. Thus $P$ satisfies property \ref{cond:post_are_leaves} of Theorem \ref{thm:P_not_renorm_exp_growth}.


The polynomial $P$ is in fact a \emph{Shabat} (or \emph{Bely\u{\i}})
  \emph{polynomial}. This means that $P$ has exactly $2$ finite
critical values. The diagram in Figure~\ref{fig:P_dessin} is the
\emph{dessin d'enfant} of $P$. For a general Shabat polynomial,
its dessin d'enfant is obtained as the preimage of an arc
connecting its two finite critical values. 
It is evident that the dessin d'enfant of any Shabat
polynomial is a tree. Conversely, every planar tree $\TC$ 
is the dessin d'enfant of a Shabat polynomial, see \cite[Theorem~2.2.9]{LanZvo}.
Roughly speaking, one considers a
  set of half-planes
  \[\bigsqcup_{\text{edges $e$ of $\TC$}} \{\Halb^+_e, \Halb^-_e\},\]
  that is, two half-planes for each edge $e$ of the tree
  $\TC$. Then, one constructs an abstract Riemann surface
  $\Delta'$ by gluing these half-planes together as indicated by
  the structure of $\TC$.  There is a natural holomorphic map
  $g\colon \Delta'\to \C$ that maps the half-planes in $\Delta'$
  alternatingly to the upper and lower half-planes in $\C$. Using
  the uniformization theorem one obtains a conformal map
  $\varphi\colon \Delta'\to \C$. The map $g\circ \varphi^{-1}$ is
  then the desired map, which can be shown to be a Shabat
  polynomial. Put differently, the construction of a Shabat
  polynomial is very similar to the construction of the rational
  maps in Section~\ref{sec:constr-f} and
  Section~\ref{sec:sierp-carp-rati}. The half-planes
  $\Halb^{+}_e$ and $\Halb^{-}_e$ correspond to the white and
  black $1$-tiles 
  from which
  $\Delta'$ was constructed there.  Finally, we note that a
  general dessin d'enfant is not necessarily planar or a
  tree. The concept was introduced by Grothendieck in \cite{Gro}
  as a way to describe algebraic curves.

\section{$P$ is not renormalizable}
\label{sec:p-not-renorm}

Here, we show that the polynomial $P$ from the previous section
is not renormalizable. That is, we verify that $P$ satisfies
property \ref{cond:non_renorm} of Theorem \ref{thm:P_not_renorm_exp_growth}. This means $P$ is a counterexample to the
the conjecture stated in Section~\ref{sec:introduction}. However the conjecture may be still valid for quadratic polynomials. We first recall
some relevant definitions.

A continuous map $f\colon U\to V$ is said to be \emph{proper} if
$f^{-1}(K)$ is compact in $U$ for every compact $K\subset V$. 
Assume now in addition that $U$ and $V$ are Jordan
domains in $\C$ and that $f$ is holomorphic. In this case $f$
extends continuously to $\partial U$ and this extension
(still denoted by $f$) satisfies $f(\partial U) = \partial
V$. The map $f\colon \partial U \to \partial V$ is
topologically 
conjugate to $z^d\colon S^1\to S^1$ for a $d\in \N$ (here we view
the unit circle $S^1$ as the boundary of the unit disk
$\D$). Furthermore, every point in $V$ has exactly $d$ preimages
in $U$ when counted with multiplicity (see \cite[Problem~15-c]{Milnor_Book}). The number $d$ is then called the \emph{degree} of the proper map $f\colon U\to V$.

\begin{definition}
  \label{def:poly-like}
  A \emph{polynomial-like map of degree} $d\geq 1$ is a triple $(f,U,V)$, where
  $U,V\subset \C$ are Jordan domains such that $\overline{U}$ is a compact subset of $V$ and $f\colon U\to V$ is a proper holomorphic map of degree $d$.  
\end{definition}

 \hide{It follows that neither $U$ nor $V$ is equal to $\C$ in
   the above definition.} 
Polynomial-like maps were introduced by
 Douady and Hubbard in \cite{DH_poly_like}. The above definition
 differs slightly from the typical one found in the literature,
 as we allow the case $d = 1$. The \emph{filled Julia set} of a
 polynomial-like map $(f,U,V)$, denoted by $\mathcal{K}(f|U)$, 
is the set of points in $U$ that never leave $U$ under iteration of $f$, that is,
\begin{equation*}
  \mathcal{K}(f|U) = \bigcap_{n\geq 0} (f|U)^{-n}(U).
\end{equation*}
It is a compact subset of
$U$. 

\begin{definition}
  \label{def:remorm}
  Let $P\colon \C \to \C$ be a polynomial of degree $d\geq 2$
  with connected Julia set. Let $n\in \N$, then $P^n\colon U\to V$ is called a
  \emph{renormalization} of $P$ if
  \begin{enumerate}
  \item 
    \label{item:remorm1}
    $(P^n,U,V)$ is a polynomial-like map;
  \item 
    \label{item:remorm2}
    $\mathcal{K}(P^n|U)$ is connected;
  \item 
    \label{item:remorm3}
    the degree $\delta$ of $(P^n,U,V)$ satisfies
    \begin{equation*}
      2\leq\delta< d^n.
    \end{equation*}    
  \end{enumerate}
If such a renormalization $P^n\colon U\to V$ exists for some $n\in\N$, then we call the polynomial $P$ \emph{renormalizable}. 
\end{definition}

Condition \eqref{item:remorm3} in the above definition excludes
 trivial polynomial-like restrictions, that is, $P$
itself or maps of degree $1$. Renormalization has been mostly considered in the case of
quadratic polynomials, see for example \cite{McM_renorm}. The
higher 
degree case has been considered in \cite{Inou}, see also
\cite{RussellDierk}. Note that in the non-quadratic case there
are several distinct definitions of renormalization. We have
chosen the most general one.


\begin{theorem}
  \label{thm:non_ren_3post}
  Let $P\colon \C\to \C$ be a polynomial with exactly two
  finite postcritical points that does not have finite periodic
  critical points. Then $P$ is not renormalizable.
\end{theorem}

For the proof we use the following 
elementary fact that is a consequence of the Riemann-Hurwitz formula \cite[Theorem 7.2]{Milnor_Book}. Let
$V\subset \CDach$ be a Jordan domain that contains a single
postcritical point $p$ of a polynomial $P$. Then every component of $P^{-1}(V)$ is a
Jordan domain that contains exactly one point from
$P^{-1}(p)$. 
\begin{proof}
  Let $P$ be a polynomial as in the statement. 
  Assume that $P$ is renormalizable, that is, there exists a
  polynomial-like map $(P^n,U,V)$ as in
  Definition~\ref{def:remorm} for some $n\in\N$. Denote by
  $\delta$ its degree. Let $p$ and $q$ be the two finite
  postcritical points of $P$ and $\KC(P^n|U)$ be the filled Julia
  set of $(P^n,U,V)$.

  Clearly, if $V$ contains no postcritical point of $P$, then the
  polynomial-like map $(P^n,U,V)$ has degree $\delta=1$ by the Riemann-Hurwitz formula. Thus $(P^n,U,V)$ is not a renormalization of $P$.
 
  Without loss of generality, we may assume from now on that
  $p\in V$. Two cases remain, namely $q\in V$ and $q\notin V$.

  First suppose that $q\in V$. This means that 
  \begin{equation*}
    \{p,q\} 
    = 
    \post(P)\setminus\{\infty\} 
    =
    \post(P^n)\setminus\{\infty\} 
    \subset 
    V. 
  \end{equation*}
  Thus the closed Jordan domain $A_\infty\coloneqq\CDach\setminus
  V$ contains a single postcritical point of $P^n$,
  namely $\infty$, in its interior. 
  Note
  that every component of $P^{-n}(A_\infty)$ contains a preimage
  of $\infty$, hence there is only one such component. Using the
  remark before the proof, we conclude that
  $P^{-n}(A_\infty) = \CDach \setminus 
  P^{-n}(V)$ is a Jordan domain. It follows that $P^{-n}(V)$
  consists of a single component, namely $U$. This in turn
  implies that the degree of the polynomial-like map
  $(P^n,U,V)$ is $\delta= \deg(P^n)= \deg(P)^n$. This contradicts
  condition~\eqref{item:remorm3} in Definition~\ref{def:remorm},
  meaning that $(P^n,U,V)$ is not a renormalization of~$P$.  

  \smallskip
  Now suppose that $q\notin V$, that is, $V$ contains exactly one
  postcritical point, namely $p$. By the remark after Theorem
  \ref{thm:non_ren_3post}, $U$ contains exactly one point in
  $P^{-n}(p)$, which we denote by $c$. Then $c$ is a critical point of $P^n$; for otherwise the degree of $(P^n,U,V)$ is $\delta=1$, contradicting condition~\eqref{item:remorm3} in Definition~\ref{def:remorm}. Since $P$ has no periodic critical points, $c\neq p$. Thus the set
  $(P^n|U)^{-1}(c)= (P^n|U)^{-2}(p)$ contains exactly $\delta$
  distinct points $c_1,\dots,c_\delta$. It follows that
  $(P^n|U)^{-1}(U)=(P^n|U)^{-2}(V)$ has exactly $\delta$
  components $U_1,\dots,U_\delta$ (where $c_j \in U_j$ for
  $j=1,\dots,\delta$). Furthermore, $\overline{U_j}\subset U$ for each $j$, since $\overline{U}\subset V$ and $P^n:U\to V$ is proper. Iterating this process, we see that each
  $U_j$ contains $\delta$ components of $(P^n|U)^{-2}(U)$, and so
  on. It follows that $U_j \cap \mathcal{K}(P^n|U) \neq \emptyset$
  for all \linebreak $j=1,\dots,\delta$. Hence
  $\mathcal{K}(P^n|U)$ is disconnected, contradicting condition \eqref{item:remorm2} in Definition \ref{def:remorm}. Thus $(P^n,U,V)$ is not a
  renormalization of~$P$.
\end{proof}

\longhide{
It is possible to show in the last case of the proof above that
$\mathcal{K}(P^n|U)$ is a Cantor set. Indeed, as a postcritically-finite polynomial, $P$ is
\emph{subhyperbolic}. Then $P$ is \emph{expanding} with respect
to the canonical orbifold metric on a neighborhood of the Julia
set $\JC_P$ (see \cite[Chapter 19, Theorem 19.6]{Milnor_Book} or
\cite[Proposition~A.33]{THEbook}).
 }

Theorem \ref{thm:non_ren_3post} implies that the polynomial $P$
from Section \ref{sec:non-renorm-polyn} is not
renormalizable. Thus $P$ satisfies property \ref{cond:non_renorm}
of Theorem~\ref{thm:P_not_renorm_exp_growth}. This finishes the
proof of Theorem~\ref{thm:P_not_renorm_exp_growth}.

\section{Appendix A}

In the literature, one usually studies properties of a self-similar group $G$ via its wreath recursion. In particular, see the proof of exponential growth of the Basilica group, that is, $\img(z^2-1)$, in \cite[Proposition~4]{GriZ_Basilica} or the study of iterated monodromy groups of postcritically-finite quadratic polynomials in \cite{BarNekr_Quad}. The goal of this appendix is to show how these techniques apply in our setting, that is, when $G$ is the iterated monodromy group of a Thurston map.

\subsection{Appendix A.1: Actions on rooted trees and self-similar groups}
\label{app:self_sim_gr}
For the convenience of the reader we review the definition of
self-similar groups and some closely related notions, more
details can be found in \cite[Chapters~1 and 2]{Nekra}.  

Choose an \emph{alphabet}
$X$ of $d$ letters. The set of all words in $X$ of length $n\in
\N$ is denoted by $X^n$. The empty word is the only word of
length $0$ and denoted by $\emptyset$. Consequently we set
$X^0=\{\emptyset\}$.
Let $X^*=\bigcup_{n\in\N_0} X^n$ be the set of all finite
words in the alphabet $X$. Then $X^*$ has a natural structure
of a $d$-ary rooted tree: we define the root to be $\emptyset$ and connect every word $v\in X^n$ to all words of the form
$vx\in X^{n+1}$ for an arbitrary letter $x\in X$ and each $n\in\N_0$. The set $X^*$ viewed as a
rooted tree is called the \emph{tree of words} in the alphabet
$X$ and is denoted by $T$. It is evident that the $n$-th level of the tree $T$ is given by the words in $X^n$, $n\in\N_0$.


Let $\Aut(T)$ be the automorphism group of the tree $T$, that is, the group of all bijective maps $g:T\to T$ that preserve the adjacency of the vertices of $T$. We consider the \emph{right action} of $\Aut(T)$ on the tree $T$. So, the image of a vertex $v$ under the action of an element $g\in \Aut(T)$ is denoted by $v^g$, and in the product $g_1 g_2$ the element $g_1$ acts first. 

\begin{definition}\label{def:level_stab}
Let $G$ be a subgroup of $\Aut(T)$. The \emph{$n$-th level stabilizer} is the subgroup $\Stab_G(n)$ of those elements of $G$ that fix pointwise all vertices of the $n$-th level $X^n$ of the tree $T$. That is,
\[\Stab_G(n)=\{g\in G: v^g = v \text{ for all } v\in X^n\}.\]
The $n$-th level stabilizer is a normal subgroup in $G$ of finite index for all $n\in \N_0$.
\end{definition}

Let $v\in X^*$ be an arbitrary vertex of the tree of words
$T$. Denote by $T_v$ the subtree of $T$ rooted at $v$ such that
the vertex set of $T_v$ is $\{vx: x\in X^*\}$. Clearly, $T_v$ is
isomorphic to the whole tree of words $T$ via the shift $\iota_v\colon T_v \to T$ defined by $vu \mapsto u$ for $u\in X^*$. 

\hide{Let $g$ be an element of the automorphism group $\Aut(T)$ of the tree $T$. It is evident, that the action of $g$ descends to the $n$-th level of $T$, that is, to the set $X^n$. }
For every $g\in\Aut(T)$ and $v\in X^*$, we define an automorphism $g|_v: T \to T$, called the \emph{restriction} of $g$ to the subtree $T_v$, by 
\[g|_v = \iota_{v^g} \circ g \circ \iota^{-1}_v.\]

To simplify the notation assume that $X =\{1, . . . , d\}$. Denote by $\Sigma(X)$ the symmetric group of permutations of the set $X$. Then every element $g\in\Aut(T)$ can be written in the following form,  called the \emph{wreath recursion} of $g$,
\[ g=\llangle g|_{1}, \dots, g|_{d}\rrangle \sigma_g, \]
where $\llangle g|_{1}, \dots, g|_{d}\rrangle \in  \Stab_{\Aut(T)}(1)\cong\Aut(T)^X$, that is, an element of the direct product $\Aut(T)\times \cdots \times\Aut(T)$ with $d=\# X$ factors, and $\sigma_g\in \Sigma(X)$ is the permutation equal to the action of $g$ on $X^1$. Formally, there is a canonical isomorphism
\begin{equation}
  \label{eq:wreath_isom}
\psi: \Aut(T) \to \Aut(T)^X \rtimes \Sigma(X),
\end{equation}
where the semidirect product is taken with respect to the natural action of $\Sigma(X)$ on the factors of $\Aut(T)^X$ (an element $\sigma\in\Sigma(X)$ acts on $\Aut(T)^X$ by the permutation of the factors coming from its action on $X$, that is, $(g_1, \dots, g_d)^\sigma=(g_{1^\sigma}, \dots,g_{d^\sigma})$).  In other words, the automorphism group $\Aut(T)$ is isomorphic to the \emph{permutational wreath  product} $\Aut(T)\wr \Sigma(X)$. That is, the wreath recursions of arbitrary elements $g,h\in\Aut(T)$ are multiplied according to the rule
\[\displaystyle\llangle g_{1}, \dots, g_{d}\rrangle \sigma \cdot \llangle h_{1}, \dots, h_{d}\rrangle \tau = \llangle g_{1}h_{1^{\sigma}}, \dots, g_{d}h_{d^{\sigma}}\rrangle \sigma \tau.\]

\begin{definition}\label{def:self-similar_group}
A group $G<\Aut(T)$ acting faithfully on the tree $T=X^*$ is said to be \emph{self-similar} if for every $g\in G$ and every $x\in X$ there exist $h\in G$ and $y\in X$ such that 
\[
(xv)^g=y(v)^h
\]
for all $v\in X^*$. Put differently, $G$ is called self-similar if each restriction $g|_v$ belongs to $G$ for all $g\in G$ and $v\in X^*$.
\end{definition}

It is clear from the definition that a group $G<\Aut(T)$ is self-similar if and only if $g|_x\in G$ for all $x\in X^1$ and all generators $g$ of $G$. 
Furthermore, every self-similar group $G$ has an \emph{associated wreath recursion}, that is, a homomorphism $\psi:G\to G\wr\Sigma(X)$ given by the restriction of the canonical isomorphism \eqref{eq:wreath_isom} to $G$.


\begin{definition}
Let $G<\Aut(T)$ be a self-similar group. $G$ is said to be \emph{recurrent} (or \emph{self-replicating}) if its action is transitive on the first level $X^1$ of the tree $T=X^*$ and for some (and thus for all) $x \in X$ the homomorphism $\psi_x:G_x\to G$ given by $g\mapsto g|_x$ is onto, where $G_x=\{g\in G: x^g=x\}$ is the stabilizer of $x$ in $G$. 
\end{definition}

Note that if a self-similar group is recurrent, then it is transitive on every level of the tree $T$ (the group is then called \emph{level-transitive}).

\begin{definition}
A self-similar group  $G<\Aut(T)$ is called \emph{regular branch} if there exists a finite index subgroup $H$ of $G$ such that $H^X < \psi(H)$, where $\psi:G \to G \wreath \Sigma(X)$ is the wreath recursion associated with $G$. In such a case we say that $G$ is regular branch over $H$.
\end{definition}


\longhide{
 \reminder{This implies that there are non-trivial rigid stabilizers of any vertex.}
 
\begin{Reminder}\mbox{}

Let $G$ be a group acting on a rooted $d$-ary tree $T^d$. For a vertex $v$ of $T^d$, let $T_v$
denote the subtree hanging down at vertex $v$ and for an element $g\in G$ let $\supp(g)$
be the \emph{support} of $g$, that is, the set of vertices not fixed by $g$. Also, for a vertex $v$ and an element $g\in G$ we denote by $g_v$ the restriction of $g$ on $T_v$.

The \emph{stabilizer of a vertex} $v$ is the subgroup $G_v = \{g \in G | g(v) = v\}$. The \emph{rigid stabilizer of a vertex} $v$ is the subgroup $G[v] = \{g \in G | \supp(g) \in T_v \}$. 

The \emph{stabilizer of level} $n$ is the subgroup $$\Stab_G(n)=\{g \in G | g(v) = v \text{ for all } v \text{ with } |v|=n\}=\bigcap_{|v|=n} G_v.$$ The \emph{rigid stabilizer of level} $n$ is the subgroup $$\Rist_G(n) = \langle G[v] | \quad |v| = n\rangle.$$ Since rigid stabilizer of distinct vertices of the same level commute, we have $$\Rist_G(n) = \prod_{|v|=n} G[v].$$

\begin{definition}
Let $G$ be group of automorphisms of a rooted tree $T^d$. $G$ is said to be a \emph{near branch group} (resp. \emph{weakly near branch group}) if for all $n\geq 1$, the subgroup $\Rist_G(n)$ has finite index in $G$ (resp. is nontrivial). If in addition $G$ acts level transitively (that is, transitively on each level of the tree) than $G$ is called a \emph{branch group} (\emph{weakly branch group}) respectively.
\end{definition}

\begin{definition}
Let $G$ be a self-similar group of automorphisms of a $d$-ary rooted tree. $G$ is said to be \emph{self-replicating} if for all $g\in G$ and all $x \in \{1, 2, \dots, d\}$, there exists an element $h\in \Stab_G(1)$ such that $h_x = g$.
\end{definition}

\end{Reminder}
}

\subsection{Appendix A.2: Further properties of $\img(f_1)$}
\label{app:alg_prop_snow}
Here we revisit the iterated monodromy group of the map $f_1$
from Section~\ref{sec:constr-f}. Let $a,b,c$ be the generators of $\img(f_1)$ as in Section \ref{sec:tiles-flow-iter}. That is, $\img(f_1)$ acts on the dynamical preimage tree $\bT_{f_1}$ that is identified with $\{1,\dots,6\}^*$ and the wreath recursions of the generators are given by \eqref{eq:wreath_f}.
We use this to prove additional properties of $\img(f_1)$. The
discussion is kept rather short, to avoid excessive details. 


We start with giving alternative proofs of Corollary
\ref{cor:ab4_inf_order} and Lemma \ref{lem:free_semi_gp} (which implies exponential growth of
$\img(f_1)$). 

\begin{proof}

To save space we will only prove that $x\coloneqq ab^{4}$ and $y\coloneqq  ab^{12}$ generate a free semigroup in $\img(f_1)$ (the proof is analogous if we introduce the third generator $z\coloneqq ab^{20}$).


Using \eqref{eq:wreath_f} one can compute the following identities
\begin{align}
  \label{eq:wreath_semigroup}
  x^{2}&=\llangle x,\cdot,x^s,\cdot,\cdot,\cdot \rrangle \;
  \\
  \notag
  y^{2}&=\llangle y,\cdot,y^s,\cdot,\cdot,\cdot \rrangle \;
  \\
  \notag
  xy&=\llangle y,\cdot,x^s,\cdot,\cdot,\cdot \rrangle \;
  \\
  \notag
  yx&=\llangle x,\cdot,y^s,\cdot,\cdot,\cdot \rrangle,
\end{align}
where $s=ab^{-1}$ and $\cdot$ represents the omitted terms.

\longhide{
\begin{align}
  \label{eq:wreath_semigroup}
  x^{2n}&=\llangle x^n,\cdot,(x^n)^t,\cdot,\cdot,\cdot \rrangle \;
  \\
  \notag
  y^{2n}&=\llangle y^n,\cdot,(y^n)^t,\cdot,\cdot,\cdot \rrangle \;
  \\
  \notag
  z^{2n}&=\llangle z^n,\cdot,(z^n)^t,\cdot,\cdot,\cdot \rrangle \;
  \\
  \notag
  xy&=\llangle y,\cdot,(x)^t,\cdot,\cdot,\cdot \rrangle \;
  \\
  \notag
  xz&=\llangle z,\cdot,(x)^t,\cdot,\cdot,\cdot \rrangle \;
  \\
  \notag
  yx&=\llangle x,\cdot,(y)^t,\cdot,\cdot,\cdot \rrangle \;
  \\
  \notag
  yz&=\llangle z,\cdot,(y)^t,\cdot,\cdot,\cdot \rrangle \;
  \\
  \notag
  zx&=\llangle x,\cdot,(z)^t,\cdot,\cdot,\cdot \rrangle \;
  \\
  \notag
  zy&=\llangle y,\cdot,(z)^t,\cdot,\cdot,\cdot \rrangle.
\end{align}
}

The first two identities imply that $x$ and $y$ are elements of
infinite order in $\img(f_1)$. Indeed, the order of $x$ cannot be
an odd number, because $x$ acts as the permutation $(1 3) (2 5) (4 6)$ on the first level of the dynamical preimage tree $\bT_{f_1}$. If $\ord(x)=2n$, $n\in \N$, then \eqref{eq:wreath_semigroup} imply that $x^n = 1$, contradiction. The same argument applies to $y$.

Now let $w$ be an arbitrary word in $x$ and $y$. Here and throughout the proof, $|w|$ denotes the length of the word $w$ with respect to the alphabet $\{x,y\}$. If $|w|$ is even, $w$ can be written uniquely as a product of $x^2$, $y^2$, $xy$, and $yx$. From (\ref{eq:wreath_semigroup}) it follows that $\psi_1(w)$ is a word in $\{x,y\}$ and $|\psi_1(w)|=|w|/2$, where $\psi_1$ denotes the projection onto the first coordinate in the canonical isomorphism \eqref{eq:wreath_isom}, that is, $\psi_1(w)=w|_1$. Moreover, $\psi_1(w)$ ends with the same letter ($x$ or $y$) as $w$.

Consider now any two distinct words $w_1$ and $w_2$ in the alphabet $\{x,y\}$. We are going to show that they represent different group elements in $\img(f_1)$. The proof is by induction on $|w_1|$+$|w_2|$.

The base cases $|w_1|+|w_2|\in\{1,2\}$ can easily be
verified. Now let $|w_1|+|w_2| \geq 3$. By multiplying from the
right with the inverse of the common ending word, it is enough to
assume that $w_1$ and $w_2$ end with different letters. We can
further assume that the parity of $|w_1|$ and $|w_2|$ is the same
(for otherwise, $w_{1}w_{2}^{-1}$ does not fix the first
level). 
If $|w_i|$ is odd set $w_{i}'=xw_{i}$, otherwise set
$w_{i}'=w_{i}$, for $i=1,2$. Then $\psi_1(w_{1}')$ and $\psi_1(w_{2}')$ are two distinct words in $\{x,y\}$. Indeed, they end with different letters since $w'_1$ and $w'_2$ do.
Furthermore, 
\begin{equation*}
\label{eq:prod_of_wr_elem}
|\psi_1(w_{1}')|+|\psi_1(w_{2}')| \leq \frac{|w_1|+1}{2}+\frac{|w_2|+1}{2} < |w_1|+|w_2|.
\end{equation*}
Thus the induction hypothesis applies to $\psi_1(w_{1}')$ and $\psi_1(w_{2}')$, which finishes the proof.
\end{proof}

\longhide{
Note that $T=[1,\infty] \cup [-\infty, -1] \subset \R$ is an
$f$-invariant tree containing $\post(f)$ (we have $f(T)=
[1,\infty]$). Let $a$ be the loop around $1$, and $b$ be the
loop around $-1$ (both mathematically positively oriented). The
wreath recursions are
\begin{align}
  \label{eq:wreath_f_triag_snow}
  a&= \llangle a, a^{-1}, 1, b, b^{-1},1 \rrangle (2356)
  \\
  \notag
  b&= (123) (456).
\end{align}

Let $G=\img(f)=\langle a,b\rangle$.

\begin{itemize}
\item $a^{24}=b^3=c^2=1$, where $c=(ab)^{-1}$.

\item $\Stab_1(G)=\langle S \rangle$, where $S=\{(a^4)^g: g\in\{1, b, b^{-1}, ab^{-1}, a^{-1}b, a^{-2}b\}\}$.

\item $a^4 c$ is an element of infinite order.

\item $G$ is regular branch over $H=\langle b, a^2 \rangle^G = \langle b^g, (a^2)^g: g\in G\rangle $ (for this one considers $a^4$ and $(b^{-1}a^4 b a^{-8})^2$).
\end{itemize}
}

To simplify the notation we denote the iterated monodromy group
$\img(f_1)$ by $G$ from now on.


\begin{lemma}
$G$ is recurrent, consequently, it is level-transitive.
\end{lemma}
\begin{proof} It is evident that $G$ acts transitively on the first level of $\bT_{f_1}$, see Figure \ref{fig:Schreierf1_11}. Using \eqref{eq:wreath_f}, we check that $b, (b^4)^c \in G_1$ ($=\{g\in G: 1^g=1\}$) and $\psi_1(b)=b$ , $\psi_1((b^4)^c)=c^{-1}b^{-1}$. Since $\langle b,c^{-1}b^{-1}\rangle =G$, the statement follows.
\end{proof}

\begin{prop}
$G$ is regular branch over \[H\coloneqq \langle[b^2,c]\rangle^G = \langle [b^2,c]^g: g\in G\rangle.\]
\end{prop}
\begin{proof}
First we verify that $H^6<\psi(H)$. To this end, we use the wreath
recursions 
\eqref{eq:wreath_f} and the fact that $(bc)^2=c^3=1$ to obtain
the following identities.~\footnote{The ``FR package'', written
  by L. Bartholdi, for the computer algebra system GAP, is very
  helpful for such computations.} 
\begin{align*}
  b^{-8}&=\llangle b^{-8}, (cb)^2, (bc)^2, c^{-8}, (bc)^2, (cb)^2 \rrangle \;
  \\
  &= \llangle b^{-8}, 1, 1, c, 1, 1 \rrangle \;
  \\
  b^{8}c^{-1}b^{-4}c&=\llangle b^{9}c, (b^{-1}c^{-1})^{2}b^{-4},(c^{-1}b^{-1})^{2}cb, c^{9}b, (c^{-1}b^{-1})^{2}c^{-4}, (b^{-1}c^{-1})^{2}bc\rrangle 
  \\
  &= \llangle b^{9}c, b^{-4}, cb, b, c^{-1}, bc \rrangle \;
  \\
   (b^{8}c^{-1}b^{-4}c)^2&= \llangle (b^{9}c)^2, b^{-8}, 1, b^{2}, c, 1 \rrangle .
\end{align*}
Then we consider the following commutator
\begin{align}
  \label{eq:wreath_comm_restr}
  [(b^{8}c^{-1}b^{-4}c)^2,b^{-8}]&=\llangle [(b^{9}c)^2,b^{-8}], 1, 1, [b^2,c], 1, 1 \rrangle \;
  \\
  \notag
  &= \llangle 1, 1, 1, [b^2,c], 1, 1 \rrangle.
\end{align}
Here we use that $[(b^{9}c)^2,b^{-8}]=1$ as one can verify (At the same time, $[b^{9}c,b^{-8}]\neq 1$. For that reason $H$ is not chosen to be the commutator subgroup $[G,G]$ of $G$).

Since $G$ is recurrent, (\ref{eq:wreath_comm_restr}) implies that $1\times 1\times 1\times H \times 1\times 1 < \psi(H)$, consequently, $H^6<\psi(H)$.

Thus, we are only left to check that $H$ has finite index in $G$. By construction, $H$ is a normal subgroup of $G$. Furthermore, a generic element of the quotient  group $G/H$ can be written in the form 
\begin{equation}
\label{eq:norm_quot_form}
(b)c^{\pm1}bc^{\pm1}\dots bc^{\pm1}b^k, \quad k\in\{0,\dots,23\},
\end{equation}
where the notation ``$(b)$'' means that $b$ may or may not appear as the first letter in the word.
Since $(bc)^2=1$, we can further normalize \eqref{eq:norm_quot_form} to
\begin{equation*}
(b)c^{\pm1}b^k, \quad k\in\{0,\dots,23\}.
\end{equation*}
Hence the quotient subgroup $G/H$ is finite.
\end{proof}

We close the appendix with two corollaries of the previous proposition.

\begin{cor}
\label{inf_abel_subgr}
$G$ contains a subgroup isomorphic to $\Z^n$ for each $n\in\N$. 
\end{cor}
\begin{proof}
First, we observe that $[c,b^4]\in H$ is of infinite order. Indeed, $[c,b^4]\in \Stab_G(1)$ and $\psi_1([c,b^4])=(ab^4)^{b^{-1}}$. Now the statement follows from $H^6 < \psi(H)$.
\end{proof}

\begin{cor}
\label{finite_L_pres}
$\img(f_1)$ has a finite \emph{endomorphic presentation}, that
is, a finite recursive presentation as introduced in \cite{Bar_Branch}. However, it is not finitely presented.
\end{cor}
\begin{proof}
An immediate corollary of \cite[Theorem 1]{Bar_Branch}, since iterated monodromy groups of postcritically-finite rational maps are \emph{contracting} \cite[Theorem 6.4.4]{Nekra}.
\end{proof}

\section{Acknowledgments}
M.H. is supported by the Studienstiftung des Deutschen
Volkes and is grateful for its continuing support. 
D.M. has been 
supported by the Academy of Finland via the Centre of Excellence
in Analysis and Dynamics Research (project No. 271983).
The authors thank the University of Jyv\"{a}skyl\"{a} and the
University of California, Los Angeles for their hospitality where
a part of this work has been done during research visits of the
authors. The authors are grateful for insightful discussions with
Mario Bonk, Dzmitry Dudko, Volodymyr Nekrashevych, Kevin
Pilgrim, and Palina Salanevich. Jana Kleineberg helped prepare some of the pictures.

\smallskip

\longhide{
\section{Outtakes}
\label{sec:outtakes}

\textsf{Add to introduction:
The criterion in Theorem \ref{thm:exp_growth_gen} is used to produce more motivating examples of rational maps with iterated monodromy groups of exponential growth. Also add smth like this: we want the criterion to be as simple as possible but at the same time flexible enough to produce desired examples + we know that the conditions on the map can be relaxed. Maybe, we should add another short section with two other examples (renormalizable with a tree and not-renormalizable).
}

\smallskip

\textsf{I am going to list more general conditions below and under the current ones (just to record them). I do not think that we should include all of them. In particular, one can assume the following:
\begin{enumerate}
\item there exists an essential $g$-invariant curve $\CC_{p}$ around some $p\in\post(g)$ (up to an isotopy rel. $\post(g)$) so that $\deg(g:\CC_p\to\CC_p)\geq 2$. \\
In the case when $\CC_p$ is a simple curve, this condition says that the map $g$ is \emph{renormalizable} (in the classical sense) to some Thurston map $h$ after collapsing $\CC_{p}$ to a point. Note that maps $f_1$ and $P$ \textit{do not} satisfy this condition (with a simple curve assumption), but the Sierpinski and obstructed examples do. However, all examples \emph{do} satisfy this assumption if we do not assume that $\CC_p$ is simple, though one needs to clarify what does ``essential'' mean in this case.
\item $g(p)=p$.
\item $\deg(g,p)=1$.
\item $\exists v$ outside the renormalization domain so that $g^k(v)=p$ and $\deg(g^k,p)$ is not a multiple of $\alpha_h(p)$.
\end{enumerate}
(1-2-3-4) imply that $\img(g)$ contains a free semigroup.
}

\textsf{\textcolor{red}{M: Probably we should uniformize the style of the references in the bibliography, in particular, the authors}} 
}

\bibliographystyle{alpha}
\bibliography{main}

\end{document}